\documentclass[11pt,reqno, final]{amsart}

\usepackage{amsthm,amsmath,amssymb,booktabs,xcolor,graphicx}
\usepackage{bbm}
\usepackage{listings}
\usepackage{xcolor}
\usepackage{appendix}
\usepackage{enumerate}
\usepackage[shortlabels]{enumitem}
\usepackage{verbatim}

\usepackage[margin=1in]{geometry}
\usepackage[utf8]{inputenc}
\usepackage[T1]{fontenc}

\usepackage[textsize=small,backgroundcolor=orange!20]{todonotes}
\usepackage[colorlinks=true, allcolors=blue]{hyperref}
\usepackage{url}
\usepackage{etoolbox}
\usepackage{appendix}
\usepackage[nameinlink, noabbrev,capitalize]{cleveref}
\crefname{equation}{}{} 
\AtBeginEnvironment{appendices}{\crefalias{section}{appendix}} 

\usepackage[color]{showkeys} 

\colorlet{refkey}{orange!20}
\colorlet{labelkey}{blue!30}

\numberwithin{equation}{section}
\newtheorem{theorem}{Theorem}[section]

\newtheorem{lemma}[theorem]{Lemma}

\crefname{claim}{Claim}{Claims}

\newtheorem{corollary}[theorem]{Corollary}

\newtheorem*{question*}{Question}

\theoremstyle{definition}
\newtheorem{definition}[theorem]{Definition}

\newtheorem*{definition*}{Definition}

\theoremstyle{remark}
\newtheorem*{remark}{Remark}

\newcommand{\snorm}[1]{\lVert#1\rVert}

\newcommand{\sang}[1]{\langle #1 \rangle}

\newcommand{\mb}{\mathbb}
\newcommand{\mbf}{\mathbf}
\newcommand{\mbm}{\mathbbm}
\newcommand{\mc}{\mathcal}
\newcommand{\mf}{\mathfrak}

\newcommand{\on}{\operatorname}

\newcommand{\wt}{\widetilde}

\allowdisplaybreaks

\title{The smallest singular value of dense random regular digraphs}

\author[Jain]{Vishesh Jain}
\address{Department of Statistics, Stanford University,
Stanford, CA 94305, USA}
\email{visheshj@stanford.edu}

\author[Sah]{Ashwin Sah}
\author[Sawhney]{Mehtaab Sawhney}
\address{Department of Mathematics, Massachusetts Institute of Technology, Cambridge, MA 02139, USA}
\email{\{asah,msawhney\}@mit.edu}

\begin{document}
\begin{abstract}
Let $A$ be the adjacency matrix of a uniformly random $d$-regular digraph on $n$ vertices, and suppose that $\min(d,n-d)\ge\lambda n$. We show that for any $\kappa \geq 0$,
\[\mb{P}[s_n(A)\le\kappa]\le C_\lambda\kappa\sqrt{n}+2e^{-c_\lambda n}.\]
Up to the constants $C_\lambda, c_\lambda > 0$, our bound matches optimal bounds for $n\times n$ random matrices, each of whose entries is an i.i.d $\on{Ber}(d/n)$ random variable. The special case $\kappa = 0$ of our result confirms a conjecture of Cook regarding the probability of singularity of dense random regular digraphs.   
\end{abstract}

\maketitle

\section{Introduction}
For positive integers $d \leq n$, let $\mc{M}_{n,d}$ denote the set of all $n\times n$ matrices with entries in $\{0,1\}$ for which each row and each column sums to $d$. One may interpret an element $A \in \mc{M}_{n,d}$ either as the adjacency matrix of a $d$-regular digraph (directed graph) on $n$ labelled vertices (where self loops are allowed, but multiple edges are not allowed) or as the adjacency matrix of a $d$-regular bipartite graph on $n + n$ labelled vertices.

Recall that the smallest singular value of a real $n\times n$ matrix $A$ is defined to be
$$s_{n}(A) = \inf_{x\in \mb{S}^{n-1}}\snorm{Ax}_2,$$
where $\mb{S}^{n-1}$ denotes the unit sphere in $\mb{R}^{n}$ and $\snorm{\cdot}_2$ denotes the standard Euclidean norm on $\mb{R}^{n}$. In particular, $A$ is singular (non-invertible) if and only if $s_n(A) = 0$. 

This paper is concerned with the non-asymptotic study of the smallest singular value of a randomly chosen element of $\mc{M}_{n,d}$, in the regime where $d$ is comparable to $n$ (i.e., in the graph theoretic interpretation above, our focus is on dense digraphs). Our main result is the following. 
\begin{theorem}
\label{thm:main}
Let $\lambda \in (0,1)$. There exist constants $C_\lambda, c_\lambda > 0$, depending only on $\lambda$, for which the following holds. Let $d\leq n$ be positive integers with $\min(d,n-d) \geq \lambda n$. Then, for any $\kappa \geq 0$,
$$\mb{P}_{A \sim \mc{M}_{n,d}}[s_{n}(A) \leq \kappa] \leq C_{\lambda}\kappa \sqrt{n} + 2e^{-c_{\lambda}n},$$
where $A\sim \mc{M}_{n,d}$ denotes a uniformly chosen element of $\mc{M}_{n,d}$. 
\end{theorem}
\begin{remark}
The case $\kappa = 0$ of the above theorem shows that $\mb{P}[A \text{ is non-invertible}]\leq 2e^{-c_\lambda n}$. This confirms a conjecture of Cook \cite[Conjecture~1.7]{Cook17b}, and is optimal up to the constants $2$ and $c_\lambda$ (as can be seen by considering the probability that two rows of $A$ are identical). Moreover, up to the constants $C_\lambda, c_\lambda$, and $2$, \cref{thm:main} matches known and optimal results for random matrices, each of whose entries is an independent copy of a $\on{Ber}(d/n)$ random variable (i.e., a random variable which is $1$ with probability $d/n$ and $0$ with probability $1-(d/n)$) (cf. \cite{livshyts2019smallest, tikhomirov2020singularity}). 
\end{remark}

In the next subsection, we provide context for our work, as well as an overview of previous results. 
\subsection{Background}
\label{sub:background}
The non-asymptotic study of the smallest singular value of random matrices goes back at least to the work on numerical analysis by von Neumann and his collaborators. The starting point of our work is the seminal result of Rudelson and Vershynin \cite{rudelson2008littlewood}, showing that for an $n\times n$ random matrix $A$, each of whose entries is an independent and identically distributed (i.i.d) subgaussian random variable with mean $0$ and variance $1$, and for  any $\kappa \geq 0$,  
\begin{align}
\label{eqn:rv}
\mb{P}[s_n(A) \leq \kappa] \leq C\kappa \sqrt{n} + C \exp(-c n),
\end{align}
where the constants $C, c > 0$ depend only on the distribution of the entries. This result is optimal up to the constants $C, c$ and (again, up to constants) unified and substantially extended many previous results, such as works of Edelman \cite{edelman1988eigenvalues} and Szarek \cite{szarek1991condition} on the smallest singular value of i.i.d centered Gaussian matrices, work of Kahn, Koml\'os, and Szemer\'edi \cite{kahn1995probability} on the probability of singularity of i.i.d Rademacher (i.e., $\pm 1$ with probability $1/2$ each) matrices, and work of Tao and Vu \cite{tao2009inverse} on the smallest singular value of i.i.d Rademacher matrices. In recent years, much work has gone into relaxing the distributional assumptions in the above result of Rudelson and Vershynin; the current best result is due to Livshyts, Tikhomirov, and Vershynin \cite{livshyts2019smallest}, who establish the bound \cref{eqn:rv} assuming only that the entries of $A$ are independent, uniformly anti-concentrated, and that the expected squared Hilbert-Schmidt norm of $A$, $\mb{E}\snorm{A}_{\on{HS}}^{2}$, is $O(n^{2})$.

Progress in the non-asymptotic study of the smallest singular value of random matrices with \emph{dependent} entries has been comparatively much slower. For instance, the fact that uniformly random $n\times n$ \emph{symmetric} $\{\pm 1\}$-valued matrices are invertible with probability tending to $1$ as $n\to \infty$ was only established in 2006 by Costello, Tao, and Vu \cite{costello2006random}, even though the corresponding non-symmetric result was already known nearly 40 years before due to Koml\'os \cite{komlos1967determinant}. Similarly, for (centered) subgaussian symmetric matrices (i.e., the entries on and above the diagonal are i.i.d copies of a centered subgaussian random variable), even though it is believed that \cref{eqn:rv} should hold, the best-known analog of \cref{eqn:rv} due to Vershynin \cite{vershynin2014invertibility} has $\kappa$ in the first term replaced by the suboptimal $\kappa^{1/9}$, and $\exp(-cn)$ in the second term replaced by the suboptimal $\exp(-n^{c})$ for some small constant $c > 0$. Despite recent efforts to optimize this constant in the case of Rademacher random variables, \cite{ferber2019singularity,campos2019singularity}, the best known bound for singularity of symmetric Bernoulli matrices is $\exp(-cn^{1/2})$ \cite{campos2019singularity}, which remains far from the conjectured exponential behavior, even though in the i.i.d Rademacher case, the near-optimal bound $(1/2 + o(1))^{n}$ on the singularity probability has recently been obtained in breakthrough work of Tikhomirov \cite{tikhomirov2020singularity}).

Another popular model of random matrices with dependent entries, which has attracted considerable attention in recent years and is the subject of the present work, is the adjacency matrix of a random $d$-regular digraphs. Note that this model has the interesting feature that \emph{no} two entries are independent of each other (in contrast with random symmetric matrices, where the dependencies are localized). Works of Cook \cite{Cook17b}, Litvak et al.\ \cite{litvak2017adjacency}, and Huang \cite{huang2018invertibility} (covering complementary regimes) show that for all $3 \leq d \leq n-3$, the probability that a uniformly random element of $\mc{M}_{n,d}$ is invertible tends to $1$ as $n\to \infty$. While it was conjectured by Cook in \cite{Cook17b} that for $\min(d, n-d) = \Omega(n)$, the singularity probability should be exponentially small (this is a special case of our \cref{thm:main}), we note that none of these works show that the singularity probability (in any regime) is smaller than even $1/\sqrt{n}$. 

The smallest singular value of uniformly random elements of $\mc{M}_{n,d}$ has been considered in the works \cite{cook2019circular, LLTTY19}. With notation as in \cref{thm:main}, Cook \cite{cook2019circular} showed that $\mb{P}[s_n(A) \leq n^{-O(\log{n}/\log{d})}] \leq O(\log^{O(1)}n/\sqrt{d})$, while Litvak et al.\ \cite{LLTTY19} showed that for $C \leq d \leq n/\log^{2}n$, $\mb{P}[s_n(A) \leq n^{-6}] \leq O(\log^{2}d/\sqrt{d})$. Note that the result of Litvak et al.\ operates in a complementary regime of $d$ compared to \cref{thm:main}, whereas the result of Cook, restricted to the regime of \cref{thm:main} gives the much weaker bound $\mb{P}[s_n(A) \leq n^{-C_\lambda}] \leq O(\log^{O(1)}/\sqrt{n})$. We note that the results of \cite{cook2019circular, LLTTY19} are actually valid for more general random matrices $A - z\on{Id}$, where $z \in \mb{C}$ is a fixed complex number with $|z| \leq \sqrt{d}$; this is crucial for the application of proving a weak circular law for $\mc{M}_{n,d}$, for which the bounds in \cite{cook2019circular, LLTTY19} are sufficient. However, for many other applications, such as the study of gaps between eigenvalues \cite{ge2017eigenvalue}, delocalization of eigenvectors \cite{rudelson2016no}, and strong circular laws \cite{tao2010random}, stronger bounds such as our \cref{thm:main} are needed. So as to not overburden the presentation, we have not pursued the direction of obtaining such bounds for $A - z\on{Id}$, but anticipate that this should be possible by adding to the proof in this paper the additional notion of real-imaginary correlations \cite{rudelson2016no}. 

Finally, we mention that a couple of models of random matrices have been studied to serve as a `warm-up' for investigating uniformly random elements of $\mc{M}_{n,d}$. Of these, the most fruitful has been the model of $\{0,1\}$-valued matrices $B$ with independent rows, such that each row is drawn uniformly from $\{0,1\}^{n}$ subject to the sum of the row being exactly $d$, although note that the independence of the rows makes the study of this model quite a bit simpler (see the discussion in the next subsection). In the same regime of $d$ as in \cref{thm:main}, Nguyen \cite{nguyen2013singularity} showed that the probability of singularity of such a matrix is at most $O_{C}(n^{-C})$ for any $C > 0$, which was improved by Ferber et al.\ \cite{ferber2019counting} to $O(\exp(-n^{c}))$ for some small constant $c > 0$. For the smallest singular value, Nguyen and Vu \cite{nguyen2012circular} showed that for any $C > 0$, there exists $D > 0$ such that $\mb{P}[s_n(B)\leq n^{-D}] = O(n^{-C})$. This was improved by Jain \cite{jain2019approximate} to $\mb{P}[s_n(B)\leq \kappa] = O(\kappa n^{2} + \exp(-n^{c}))$, for some small constant $c > 0$ and for all $\kappa \geq 0$. Very recently, Tran \cite{Tra20} obtained an optimal estimate of the form \cref{thm:main} for this model; the notion of Combinatorial Least Common Denominator (CLCD) introduced in his work will be useful for us.   

\subsection{Overview of the proof} To better illustrate our ideas, we begin by briefly recalling the geometric framework of Rudelson and Vershynin \cite{rudelson2008littlewood} for controlling the smallest singular value of an $n\times n$ matrix $M$ with i.i.d sub-Gaussian entries. The unit sphere $\mb{S}^{n-1}$ is decomposed into compressible vectors (i.e., those which are close to sparse vectors), and incompressible vectors. It is not hard to show that for any unit vector $x$, $\snorm{Mx}_2 =  \Omega(\sqrt{n})$ (except with exponentially small probability); the estimate for compressible vectors then follows from the low metric entropy of the set of compressible vectors, as well as the fact that the operator norm of $M$ is $O(\sqrt{n})$ (except with exponentially small probability). For incompressible vectors, an efficient averaging procedure reduces to studying the distance of the last (say) row of the matrix to the span of the first $n-1$ rows. This amounts to studying the inner product of the last row of the matrix with a unit vector orthogonal to the span of the first $n-1$ rows. The remainder of the proof is then devoted to showing that any unit vector orthogonal to the first $n-1$ rows of the matrix is (except with exponentially small probability) arithmetically very unstructured, in the sense of having exponentially large Least Common Denominator (LCD). This is accomplished via a union bound -- we decompose the relevant range of the LCD dyadically, and note that for each dyadic interval $[D, 2D)$, the metric entropy at the relevant scale is swamped by the probability of the image of the vector under $M$ having small norm. In slightly more detail (and omitting absolute constants), for $\alpha = \mu \sqrt{n}$, where $\mu > 0$ is a small constant which can be freely chosen \emph{at the end} of the argument, it is seen that for $x \in \mb{S}^{n-1}$ with LCD in the dyadic interval $[D,2D)$, the probability that $\snorm{Mx}_2 \leq \alpha \sqrt{n}/D$ is at most $(\alpha/D)^{n-1}$. On the other hand, there is an $(\alpha/2D)$-net of such vectors of size at most $(D/\sqrt{n})^{n}$. For the relevant range of $D$, this leads to the exponential gain of $\mu^{n}$.    

Our proof of \cref{thm:main}, while broadly based on the geometric framework, encounters challenges at every step due to the lack of dependence between the entries.

\noindent \textbf{Working on $\mb{S}^{n-1}_{0}$: }In contrast to the i.i.d case, there is no uniform anti-concentration estimate available for general $x \in \mb{S}^{n-1}$ in our setting -- for instance, the inner product of any row with the vector $(1/\sqrt{n},\dots,1/\sqrt{n})$ is always $d/\sqrt{n}$. To avoid this issue, we always restrict to the part of the unit sphere orthogonal to the all ones vector (denoted by $\mb{S}_{0}^{n-1}$), noting that the smallest singular vector must always be a part of this set. Moreover, it is seen (\cref{cor:incomp-signed-spread}) that incompressible vectors in $\mb{S}_{0}^{n-1}$ have linearly many positive and negative coordinates of size $\Theta(1/\sqrt{n})$ -- this enables us to avoid explicit use of other classes of vectors, such as non almost-constant vectors in \cite{Cook17b}. 

\noindent \textbf{Refined switching: }One of the main challenges in adapting the geometric framework to our model is the lack of independence between rows. Notably, any collection of $n-1$ rows completely determines the remaining row, which precludes the use of the distance-from-hyperplane reduction in the i.i.d case. To overcome this challenge, previous works on this model (starting with the pioneering work of Cook \cite{Cook17b}) have used a `switching' operation based on the following observation. Even after conditioning on the sum of two distinct rows (say $R_k$ and $R_\ell$), there is additional randomness remaining on the set of coordinates where the sum $R_k + R_\ell$ is exactly $1$, in the following sense: for two distinct coordinates $i,j$ in this set, it is equally likely that $R_k(i) = 1, R_\ell(i) = 0, R_k(j) = 0, R_\ell(j) = 1$ or $R_k(i) = 0, R_\ell(i) = 1, R_k(j) = 1, R_\ell(j) = 0$. For our purpose, such a switching operation based on pairing entire rows is too inefficient (since it effectively increases the key probability estimate of $(\alpha/D)^{n-1}$ in the i.i.d case to approximately $(\alpha/D)^{n/2}$). Hence, we introduce a refined switching operation, which takes a `splitting set' $S$ of size $n/2$ and a permutation $\sigma$, and then (roughly) switches $R_{\sigma(2i-1)}|_{S}$ with $R_{\sigma(2i)}|_{S}$ and $R_{\sigma(2i)}|_{S^c}$ with $R_{\sigma(2i+1)}|_{S^{c}}$ -- this ensures that we have access to $n - O(1)$ independent anti-concentration events (see, e.g., \cref{eqn:anti-y}). The set $S$ and permutation $\sigma$ are chosen from a collection -- crucially of constant size (\cref{lem:robust-split-matching}) -- satisfying certain properties (see also the discussion after \cref{def:T}).  

\noindent \textbf{Quantile CLCD (QCLCD): }Our substitute for the notion of LCD is the QCLCD (\cref{def:QCLCD}), which is based on the CLCD recently introduced by Tran \cite{Tra20} with the following crucial twist: we consider the $\ell$th smallest (for $\ell = O(1)$) CLCD of a carefully chosen collection of restrictions of the vector. \cref{def:QCLCD} has some similarities to the notion of `$(t,\ell)$-bad vectors' in \cite[Definition~4.3]{ferber2019counting}, in that both definitions remove the `very worst' restrictions of a vector. However, in our application, we will be able to remove only the $O(1)$ worst restrictions, as opposed to \cite{ferber2019counting}, where the $n^{\epsilon}$ worst restrictions need to be removed. This will be crucial in proving Cook's conjecture, for which removing $\omega(1)$ restrictions is already insufficient. The main idea behind the definition of the QCLCD is the following: suppose the QCLCD of a vector (such that the restriction sets form a well-spread family (\cref{def:well-spread})) is $D$. Then, we know that all but $\ell  = O(1)$ restrictions of the vector have CLCD at least $D$, so that we will still have access to $n- O(1)$ independent anti-concentration events (heuristically, compared to the i.i.d case, we now have a term of $(\alpha/D)^{n- O(1)}$). On the other hand, the definition of a well-spread family will ensure that at least one of the (linear-sized) restrictions falls within a level set of the CLCD (\cref{def:level-set-CLCD}); this information will allow us to obtain much more efficient nets for level sets of the QCLCD (\cref{lem:HS-net-QCLCD}) (heuristically, of size $(D/\sqrt{n})^{c_\lambda n}\times (D/\alpha)^{n-c_\lambda n}$) which will enable the union bound argument in \cref{sub:small-QCLCD} to go through, since we have a gain of $\mu^{c_\lambda n}$ (compared to $\mu^{n}$ in the i.i.d case, but this is certainly enough). We note here that since the operator norm of $A$ is $d$ (which is order $n$ as opposed to order $\sqrt{n}$), the standard nets used in the i.i.d sub-Gaussian case will be insufficient, and we will instead make use of the more refined randomized-rounding based net construction due to Livshyts \cite{Liv18}.

\noindent \textbf{New quasi-randomness properties: }To execute the strategy in the previous paragraph, in particular to ensure that the set of restrictions form a well-spread family, we will need several quasi-randomness properties of random $d$-regular digraphs. Some of these are similar to those appearing in previous works \cite{Cook17b, LLTTY17}, whereas some of them are stronger. We provide a concise proof (exploiting the asymptotic enumeration of digraphs with a prescribed degree sequence due to Canfield et al. \cite{CGM08}) that a random regular digraph has these properties except with exponentially small probability (\cref{thm:quasirandom}). We note that these quasi-randomness properties are also important in our proof that for any $x\in \mb{S}^{n-1}_{0}$, $\snorm{Ax}_2 = \Omega(\sqrt{n})$ except with exponentially small probability (\cref{lem:single-vector}), which is used to handle the compressible case.  

\noindent \textbf{Partitioning the set of regular digraphs: }There is one final issue, which is that we cannot condition on the first $n-1$ rows as in the i.i.d case (since then, the last row is completely determined). Overcoming this is based on the general strategy of Litvak et al. \cite{LLTTY19}, except that we also need to incorporate arithmetic structure. Roughly, the argument proceeds as follows: given the target $\kappa$ for the smallest singular value in \cref{thm:main}, we first rule out vectors with QCLCD at most $\mu n/\kappa$ (where $\mu$ is a small constant) using a union bound argument as outlined above (compare this to the i.i.d case, where the union bound argument rules out all vectors with subexponential LCD). The remaining vectors are then assigned to constantly many partitions, based on the choice of a `splitting set' $S$ and permutation $\sigma$ as above. For the rest of this discussion, fix such a part. We are then able to use a modification of an argument of Litvak et al. \cite{LLTTY19} to reduce to the event that row $A_{\sigma(1)}$ (say) has small inner product with a vector $v$ determined only by rows $A_{\sigma(1)}+A_{\sigma(2)},A_{\sigma(2)},\dots,A_{\sigma(n)}$. Since we have already ruled out vectors with QCLCD at most $\mu n/\kappa$, the vector $v$ will be seen to have large CLCD (with respect to the randomness available by switching the relevant restrictions of $A_{\sigma(1)}, A_{\sigma(2)}$), at which point we are able to conclude.    

\subsection{Extensions}
While we have not pursued the direction of analysing the smallest singular value of complex shifts of $A$, we believe that given our general framework for handling arithmetic structure for random regular digraphs, this should be possible by adding the ingredient of real-imaginary correlations \cite{rudelson2016no}. With appropriate modifications, our methods should also extend to more general dense contingency tables. Finally, we believe that our notion of quantile-based LCDs should be generally useful in studying the smallest singular value of random matrices with dependent entries.

\subsection{Notation}
For $A \in \mc{M}_{n,d}$, we will denote rows by $A_i$ and columns by $A^{(i)}$. For an integer $N$, $\mb{S}^{N-1}$ denotes the set of unit vectors in $\mb{R}^{N}$, and $\mb{S}_0^{N-1}$ denotes the set of points $x = (x_1,\dots,x_N) \in \mb{S}^{N-1}$ such that $\sum_{i=1}^{N} x_i = 0$. Also, $B_2^N$ denotes the unit ball in $\mb{R}^{N}$ (i.e., the set of vectors of Euclidean norm at most $1$). For a matrix $A = (a_{ij}) \in \mb{R}^{N\times N}$, $\snorm{A}$ is its spectral norm (i.e., $\ell^{2} \to \ell^{2}$ operator norm), and $\snorm{A}_{\on{HS}}$ is its Hilbert-Schmidt norm, defined by $\snorm{A}_{\on{HS}}^{2} = \sum_{i,j}a_{ij}^{2}$.

We will let $[N]$ denote the interval $\{1,\dots, N\}$, $\mf{S}_{[N]}$ denote the set of permutations of $[N]$, and $\binom{[N]}{k}$ denote the set of subsets of $[N]$ of size exactly $k$. We will denote multisets by $\{\{\}\}$, so that $\{\{a_1,\dots, a_{n}\}\}$, with the $a_i$'s possibly repeated, is a multi-set of size $n$. For a vector $v \in \mb{R}^{N}$ and $T\subseteq [N]$, $v|_{T}$ denotes the $|T|$-dimensional vector obtained by only retaining the coordinates of $v$ in $T$.

We will also make extensive use of asymptotic notation. For functions $f,g$, $f = O_{\alpha}(g)$ (or $f\lesssim_{\alpha} g$ means that $f \leq C_\alpha g$, where $C_\alpha$ is some constant depending on $\alpha$; $f = \Omega_{\alpha}(g)$ (or $f \gtrsim_{\alpha} g$) means that $f \geq c_{\alpha} g$, where $c_\alpha > 0$ is some constant depending on $\alpha$, and $f = \Theta_{\alpha}(g)$ means that both $f = O_{\alpha}(g)$ and $f = \Omega_{\alpha}(g)$ hold. For  parameters $\varepsilon, \delta$, the relation $\varepsilon \ll_{\alpha} \delta$ means that $\varepsilon$ is smaller than $c_{\alpha}(\delta)$ for a sufficiently decaying function $c_{\alpha}$ depending on $\alpha$. In practice, the function $c_\alpha$ will always be polynomial with coefficients depending on $\alpha$.

All logarithms are natural, unless indicated otherwise, and floors and ceilings are omitted when they make no essential difference.
\subsection{Organization}
The remainder of this paper is organized as follows. In \cref{sec:preliminaries}, we collect some preliminaries; the main new results are \cref{lem:robust-split-matching}, \cref{thm:quasirandom}, and \cref{lem:single-vector}. In \cref{sec:rerandomization}, we introduce our refined switching technique, as well as the notion of the QCLCD, and discuss several key properties. Finally, \cref{sec:singular-value} proves \cref{thm:main}.

\section{Preliminaries}\label{sec:preliminaries}
For the remainder of this paper, we will assume that $\lambda \in (0,1/2]$. This can be done without loss of generality due to the following reason: for any $A \in \mc{M}_{n,d}$, the vector, each of whose coordinates is $1/\sqrt{n}$, is deterministically a unit vector achieving the largest singular value; hence, any vector attaining the smallest singular value of $A$ must belong to $\mb{S}_{0}^{n-1}$. Moreover, for any $x \in \mb{S}_{0}^{n-1}$ and $A \in \mc{M}_{n,d}$, we have $\snorm{ A x}_2 = \snorm {(J - A)x}_2$, where $J$ is the $n\times n$ all ones matrix. Finally, noting that $A \mapsto J - A$ is a bijection from $\mc{M}_{n,d}$ to $\mc{M}_{n,n-d}$ justifies the claim. 
\subsection{Compressibility, Almost-Constancy, and Robust Combinatorial Structures}\label{sub:robust}
We will make use of the decomposition of the unit sphere, formalized by Rudelson and Vershynin \cite{rudelson2008littlewood}, into \emph{compressible} and \emph{incompressible} vectors.
\begin{definition}
Given $\delta,\rho \in (0,1)$, we define $\on{Comp}_{\delta,\rho}$ to be the subset of $\mb{S}^{N-1}$ which is within Euclidean distance $\rho$ of a $\delta N$-sparse vector (i.e. a vector in $\mb{R}^{N}$ with at most $\delta N$ non-zero coordinates). Let $\on{Incomp}_{\delta,\rho}$ be the remaining vectors in $\mb{S}^{N-1}$.

Further, let $\on{Incomp}_{\delta,\rho}^0$ be the set of vectors $v \in \on{Incomp}_{\delta, \rho}$ satisfying $\mbf{1}\cdot v = 0$, and similarly for $\on{Comp}_{\delta,\rho}^0$.

We also define $\on{Cons}_{\delta,\rho}$ to be the set of vectors $v\in\mb{R}^{N}$ for which there exists some $\lambda\in\mb{R}$ such that $|v_i-\lambda|<\rho\snorm{v}_2/\sqrt{N}$ for at least $(1-\delta)N$ coordinates $i\in[N]$. 

We will repeatedly use these notions for restrictions of vectors, in which case the implicit dimension is modified and understood accordingly.
\end{definition}
We record some useful consequences of these definitions. 

\begin{lemma}[Incompressible vectors are spread, {\cite[Lemma~3.4]{rudelson2008littlewood}}]\label{lem:incomp-spread}
Fix $\delta,\rho \in (0,1)$. There exist $\nu_i = \nu_i(\delta,\rho) > 0$ for $i\in[3]$ such that any $v\in\on{Incomp}_{\delta,\rho}$ has at least $\nu_1N$ coordinates $i\in[N]$ with $|v_i\sqrt{N}|\in[\nu_2,\nu_3]$.
\end{lemma}
The following corollary shows that any vector in $\on{Incomp}_{\delta, \rho}^{0}$ has many positive \emph{and} negative coordinates of size $1/\sqrt{N}$. 
\begin{corollary}[Incompressible sum-zero vectors are bi-spread]\label{cor:incomp-signed-spread}
Fix $\delta,\rho \in (0,1)$. There exist $\nu_i = \nu_i(\delta,\rho) > 0$ for $i\in[3]$ such that any $v\in\on{Incomp}_{\delta,\rho}^0$ has at least $\nu_1N$ coordinates $i\in[N]$ with $v_i\sqrt{N}\in[\nu_2,\nu_3]$, and at least $\nu_1 N$ coordinates $j\in [N]$ with $v_j \sqrt{N} \in [-\nu_3,-\nu_2]$.
\end{corollary}
\begin{remark}
In particular, this shows that $\on{Incomp}_{\delta, \rho}^{0}\cap \on{Cons}_{\delta',\rho'} = \emptyset$ for $\delta', \rho' \ll \delta, \rho$. 
\end{remark}
\begin{proof}
By \cref{lem:incomp-spread}, there exist $\mu_1,\mu_2,\mu_3 > 0$ such that any $v\in\on{Incomp}_{\delta,\rho}^0$ has at least $\mu_1N$ indices $i \in [N]$ with $|v_i\sqrt{N}|\in[\mu_2,\mu_3]$. For a given $v \in \on{Incomp}_{\delta, \rho}^{0}$, assume without loss of generality that at least $\mu_1N/2$ of these are positive coordinates. In particular, the sum of the positive coordinates of $v$ is at least $\mu_1N/2\cdot\mu_2/\sqrt{N} = (\mu_1\mu_2/2)\sqrt{N}$. 

Since $\sum_i v_i = 0$ by definition, it follows that the sum of the negative coordinates of $v$ is also at least $(\mu_1\mu_2/2)\sqrt{N}$ in magnitude. Moreover, since $\|v\|_{2} = 1$, it follows that there are at most $(\mu_1^{2}\mu_2^{2}/16) N$ coordinates with value at most $- 4/(\mu_1 \mu_2 \sqrt{N})$; by Cauchy-Schwarz, the sum of the magnitudes of these coordinates is at most $(\mu_1\mu_2/4)\sqrt{N}$. Hence, the sum of the magnitudes of the coordinates which are contained in the interval $[-4/(\mu_1\mu_2\sqrt{N}),0]$ is at least $(\mu_1\mu_2/4)\sqrt{N}$. Finally, the sum of coordinates in $[-\mu_1\mu_2/(8\sqrt{N}),0]$ is at most $(\mu_1\mu_2/8)\sqrt{N}$ in magnitude, so that the sum of the coordinates in $[-4/(\mu_1\mu_2\sqrt{N}),-\mu_1\mu_2/(8\sqrt{N})]$ is at least $(\mu_1\mu_2/8)\sqrt{N}$ in magnitude. In particular, there are at least $(\mu_1^{2} \mu_{2}^{2}/32) N$ such coordinates. 

Finally, taking $\nu_1 = \min\{\mu_{1}/2, \mu_{1}^{2}\mu_{2}^{2}/32\}$, $\nu_{2} = \min\{(\mu_1 \mu_{2})/8, \mu_{2}\}$ and $\nu_{3} = \max\{\mu_{3}, 4/(\mu_1\mu_2)\}$ gives the desired conclusion.
\end{proof}

We will also use the existence of `robust splittings and matchings' of the set of coordinates $[N]$. In particular, given $\delta, \rho \in (0,1)$, we find a fixed (universal) system of $O_{\delta,\rho}(1)$ different pairs $(\sigma,S) \in \mf{S}_{[N]} \times \binom{[N]}{N/2}$ with the property that any $v \in \on{Incomp}_{\delta, \rho}^{0}$ has many of its `typical size' positive and negative elements in both $S$ and $[n]\setminus S$, and moreover, has many coordinates in consecutive positions $\sigma(i),\sigma(i+1)$ differing by order at least $1/\sqrt{N}$.
In fact, as we will see, a suitably chosen random family of pairs works well, and the justification of this fact uses no facts about sum-zero incompressible vectors except for \cref{cor:incomp-signed-spread}.

We first define the necessary events. 
\begin{definition}
\label{def:I}
Given $w\in\mb{S}^{N-1}$, $\sigma \in \mf{S}_{[N]}$, and a 3-tuple $\nu = (\nu_1, \nu_2, \nu_3) \in \mb{R}^{3}$ with $\nu_1,\nu_2,\nu_3 > 0$, we say that the event $\mc{I}_\nu(w,\sigma)$ holds if there are at least $\nu_1N$ indices $i\in[N-1]$ with $|w_{\sigma(i)}-w_{\sigma(i+1)}|\sqrt{N}\ge\nu_2$.
\end{definition}
\begin{definition}
\label{def:J}
Given $v\in\mb{S}^{N-1}$, $S\subseteq[N]$, and a 3-tuple $\nu = (\nu_1, \nu_2, \nu_3) \in \mb{R}^{3}$ with $\nu_1,\nu_2,\nu_3 > 0$, we say that the event $\mc{J}_\nu(v,S)$ holds if
\begin{enumerate}[1.]
\item there are at least $\nu_1N$ indices $i\in S$ and at least $\nu_1 N$ indices $j \in S^{c}$ with $v_i\sqrt{N}, v_j \sqrt{N}\in[\nu_2,\nu_3]$, and
\item there are at least $\nu_1N$ indices $i\in S$ and at least $\nu_1 N$ indices $j \in S^{c}$ with $v_i\sqrt{N}, v_j \sqrt{N}\in [-\nu_3,-\nu_2]$.
\end{enumerate}
\end{definition}
\begin{lemma}[A constant-sized universal family of robust combinatorial structures]\label{lem:robust-split-matching}
Fix $\delta,\rho \in (0,1)$. There exist $\nu_i(\delta,\rho) > 0$ for $i\in[3]$ and there is a family $\mc{R}_{\delta,\rho}$ of size $m_{\delta,\rho}$ of $(\sigma,S) \in \mf{S}_{[N]}\times \binom{[N]}{N/2}$  such that the following holds: for any $w,v\in\on{Incomp}_{\delta,\rho}^0$ there is $(\sigma,S)\in\mc{R}_{\delta,\rho}$ such that $\mc{I}_\nu(w,\sigma)$ and $\mc{J}_\nu(v,S)$ hold.
\end{lemma}
\begin{proof}
We will separately construct a family of $S\in\binom{[N]}{N/2}$ and a family of $\sigma\in\mf{S}_{[N]}$ with the desired properties. Then, simply taking all pairs $(\sigma,S)$ clearly satisfies the desired conclusion.

First, we find a family of sets $S$. Let $\nu_1',\nu_2',\nu_3' > 0$ be as in \cref{cor:incomp-signed-spread}. Consider $m_1$ sets chosen uniformly and independently from among all subsets of $[N]$ of size $N/2$. Denote this random collection of subsets by $\mc{R}_1$. Note that for any fixed pair of disjoint subsets $T_1,T_2 \subseteq [N]$ with $|T_1| = |T_2| = \nu_1'N$, a subset $S$ chosen uniformly from $\binom{[N]}{N/2}$ has each of $S\cap T_i$ and $S^c\cap T_i$ of size at least $\nu_1'N/3$ with probability $1-\exp(-\Omega_{\nu_1'}(N))$. Therefore, taking $m_1$ sufficiently large (in terms of $\nu_1'$, which in turn depends on $\delta,\rho$) and taking a union bound over pairs of disjoint subsets $T_1,T_2 \subseteq [N]$ with $|T_1| = |T_2| = \nu_1' N$, we find that there is a fixed family $\mc{R}_1$ of size $m_1$ with the following property: for any pair of disjoint subsets $T_1,T_2 \subseteq [N]$ with $|T_1| = |T_2| = \nu_1'N$, there is $S\in\mc{R}_1$ with each of $S\cap T_i$ and $S^c\cap T_i$ of size at least $\nu_1'N/3$. Now, since any $v\in\on{Incomp}^0_{\delta,\rho}$ has at least $\nu_1'N$ positive and negative elements of the correct size (by \cref{cor:incomp-signed-spread}), we see that for any $v \in \on{Incomp}_{\delta, \rho}^{0}$, there exists $S \in \mc{R}_{1}$ such that $\mc{J}_{\nu}(v,S)$ holds (for $\nu = (\nu'_1/3, \nu'_2, \nu'_3)$).

Next, we find a family of permutations $\sigma$. It suffices to show that there is a fixed family of permutations of $[N]$, $\mc{R}_{2}$, of size $m_2 = m_2(\delta, \rho)$ with the following property: for any pair of disjoint subsets $T_1, T_2 \subseteq [N]$ with $|T_1| = |T_2| = \nu_1' N$, there is $\sigma \in \mc{R}_{2}$ with $\sigma(i) \in T_1$ and $\sigma(i+1) \in T_2$ for at least $c(\nu_1')N$ indices $i \in [N-1]\cap (2\mb{Z}+1)$. Then, since for any $v\in \on{Incomp}_{\delta, \rho}^{0}$, any value among the $\nu_1' N$ positive elements with magnitude at least $\nu_2'/ \sqrt{N}$ differs from any value among the $\nu_1' N$ negative elements of magnitude at least $\nu_{2}'/\sqrt{N}$ by at least $2\nu_{2}'/\sqrt{N}$, we will get the desired conclusion (for $\nu = (c(\nu_1'), \nu_2', \nu_3')$). As before, it suffices to show that for a fixed pair of disjoint subsets $T_1, T_2 \subseteq [N]$ with $|T_1| = |T_2| = \nu_1' N$, the probability that a uniformly random permutation $\sigma$ satisfies $\sigma(i)\in T_1$ and $\sigma(i+1)\in T_2$ for at least $c(\nu_1')N$ indices $i\in [N-1]\cap (2\mb{Z} + 1)$ is at least $1 - \exp(-\Omega_{\nu_1'}(N))$. To see this, let $f\colon \mf{S}_{[N]}\to \mb{R}$ denote the number of such indices. Then, it follows from the linearity of expectation that $\mb{E}[f] \geq (\nu_{1}')^{2}\cdot (N-1)/2$. Moreover, it is clear that $f$ is at most $2$-Lipschitz with respect to the normalized Hamming distance on $\mf{S}_{[N]}$. Therefore, by the concentration of Lipschitz functions on the symmetric group (cf. \cite[Theorem~5.2.6]{vershynin2018high}), it follows that $\mb{P}[f \geq (\nu_1')^{2}\cdot (N-1)/4] \geq 1 - \exp(-\Omega_{\nu_1'}(N))$, as desired.  \qedhere

\end{proof}

\subsection{Combinatorial LCD}\label{sub:CLCD}
For quantifying the arithmetic structure of vectors, it will be convenient to use the notion of combinatorial least common denominator (CLCD), recently introduced by Tran \cite{Tra20} in his work on the least singular value of random zero/one matrices, each of whose rows sums to $n/2$.
\begin{definition}
[Combinatorial Least Common Denominator (CLCD), {\cite[Definition~1.4]{Tra20}}]
\label{def:CLCD}
For a vector $v\in\mb{R}^N$, $\gamma \in (0,1)$, and $\alpha > 0$, we define
\[\on{CLCD}_{\alpha,\gamma}(v) = \on{LCD}_{\alpha,\gamma}(D(v)) = \inf\{\theta > 0: \on{dist}(\theta D(v),\mb{Z}^{\binom{N}{2}}) < \min(\gamma|\theta D(v)|,\alpha)\},\]
where $D(v)$ is the vector in $\mb{R}^{\binom{N}{2}}$ with coordinates $v_i-v_j$ for $i < j$.
\end{definition}
\begin{remark}
We will take $\gamma\in(0,1)$ of constant order and $\alpha$ of order linear in $N$, in a similar manner to Tran \cite{Tra20}. For `typical' vectors, the CLCD will be at least $\sqrt{N}$ in size as with the usual LCD; see \cref{lem:lower-CLCD}. Also, note that scaling a vector down by a multiplicative factor will scale the CLCD up by the same factor.
\end{remark}
In order to state the key property of CLCD, we first define the L{\'e}vy concentration function of a random variable $X$.
\begin{definition}\label{def:levy-concentration}
For a random variable $X$ and $\epsilon \geq 0$, the L{\'e}vy concentration of $X$ of width $\epsilon$ is 
\[\mc{L}(X,\epsilon) = \sup_{x\in\mb{R}}\mb{P}[|X-x| < \epsilon].\]
\end{definition}

The key properties of the CLCD (analogous to standard properties of the LCD from \cite{rudelson2008littlewood}) are the following results from Tran \cite{Tra20}. 
\begin{definition}
\label{def:W}
Given a vector $v \in \mb{R}^{N}$ and $t \in [N]$, we define the random variable $W_{t,v}$ as $W_{t,v}:= \sum_{i=1}^N b_iv_i$, 
where $b = (b_1,\dots, b_N)$ is a uniformly random vector on the $\{0,1\}$-Boolean hypercube summing to $t$.
\end{definition}
\begin{lemma}[Anti-concentration via CLCD]\label{lem:levy-concentration-CLCD}
For any $a > 0$ and $\gamma\in(0,1)$, there exists $C = C(a, \gamma)$ depending only on $a,\gamma$ for which the following holds. Let $v\in\mb{R}^N$ with $\snorm{D(v)}_2\ge a\sqrt{N/(t(1-t))}$. Then, for every $\alpha > 0$ and $\epsilon\ge 0$,
\[\mc{L}(W_{tN,v},\epsilon)\le C\epsilon + \frac{C}{\on{CLCD}_{\alpha,\gamma}(v)} + Ce^{-8t(1-t)\alpha^2/N}.\]
\end{lemma}
\begin{proof}
This follows from \cite[Theorem~3.2]{Tra20} in the same way as \cite[Theorem~1.5]{Tra20}.
\end{proof}

The next lemma provides a useful lower bound on the CLCD of vectors which are not almost-constant. 
\begin{lemma}[Non almost-constant vectors have large CLCD, {\cite[Lemma~2.15]{Tra20}}]\label{lem:lower-CLCD}
Let $\delta,\rho\in(0,1)$ and let $v\in\mb{R}^{N-1}\setminus\on{Cons}_{\delta,\rho}$. Then for every $\alpha > 0$ and every $\gamma\in(0,\delta\rho/12)$, we have
\[\on{CLCD}_{\alpha,\gamma}(v)\ge\frac{1}{7\snorm{v}_2}\sqrt{\delta N}.\]
\end{lemma}
\begin{remark}
The version in \cite{Tra20} is stated only for $\snorm{v}_2=1$, but the statement above is an easy consequence. 
\end{remark}
Next, we need that the $\on{CLCD}$ of a vector is `approximately stable' under small Euclidean perturbations.
\begin{lemma}[Stability of the CLCD, {\cite[Lemma~2.14]{Tra20}}]\label{lem:stability-CLCD}
Let $v\in\mb{R}^N$, $\alpha > 0$, and $\gamma\in(0,1)$. Then, for any $w\in\mb{R}^N$ with $\snorm{v-w}_2<\gamma\snorm{D(v)}_2/(5\sqrt{N})$, we have
\[\on{CLCD}_{\alpha/2,\gamma/2}(w)\ge\min\bigg(\on{CLCD}_{\alpha,\gamma}(v),\frac{\alpha}{4\sqrt{N}\snorm{v-w}_2}\bigg).\]
\end{lemma}
Finally, we need a result on the metric entropy of level sets of the CLCD. This result is essentially stated in Tran \cite{Tra20}, except that we allow the length of the vectors to vary in an interval of constant order (rather than be constrained to live on the unit sphere as in \cite{Tra20}). A trivial modification of the argument in \cite{Tra20} produces the required result, so we do not provide a detailed justification here.
\begin{definition}[Level sets of CLCD]\label{def:level-set-CLCD}
Let $H > 0$ and $\chi, \mu \in (0,1)$. We define
\[L_{H,\chi, \mu} = \{x\notin\on{Cons}_{\delta,\rho}: \snorm{x}_2\in[\chi,1], H\le\on{CLCD}_{\mu N,\gamma}(x)\le 2H\}.\]
\end{definition}
\begin{lemma}[Nets of level sets of CLCD, From {\cite[Lemma~2.19]{Tra20}}]\label{lem:metric-entropy-CLCD}
Assume that $0 < \delta,\rho\ll 1$ and $0 < \mu\ll \zeta \ll_{\delta,\rho}\gamma\ll_{\delta,\rho}1$. Fix $H\ge \zeta \sqrt{n}$. Then, there exists a $(9\mu\sqrt{N}/H)$-net $\mc{N}$ of $L_{H,\chi, \mu}$ of cardinality at most $\mu^{-3}H^3(C_{\delta,\rho,\gamma,\chi,\zeta}H/\sqrt{N})^N$.
\end{lemma}
\begin{remark}
The key point in the lemma is that the constant $C_{\delta, \rho, \gamma, \chi, \zeta}$ is independent of $\mu$, so that there is no $\mu$ dependence in the base of the exponent $N$. The extra $\mu^{-3}H^3$ in net size comes from a slight but unimportant technical error in the presentation of \cite{Tra20}, as well as buffer for our version which is applicable to vectors not necessarily on the unit sphere. Also, the condition on $H$ here is slightly weaker than in \cite{Tra20}, but is proved in an identical fashion.
\end{remark}

\subsection{Quasirandomness Properties of \texorpdfstring{$d$}{d}-Regular Digraphs}\label{sub:quasirandom}
We will need various quasirandomness properties of $d$-regular digraphs, which are concisely captured in \cref{thm:quasirandom}. In our regime $d = \lambda n$ (for fixed $\lambda \in (0,1/2]$), these are straightforward consequences of the asymptotic enumeration of digraphs with specified degree sequences, due to Canfield, Greenhill, and McKay \cite{CGM08} (building on seminal work of McKay and Wormald \cite{MW90}, which solved the analogous problem for graphs). The techniques in \cite{CGM08, MW90} represent the number of digraphs with prescribed degree sequences as a contour integral, and then analyze the resulting expression using saddle points -- in our case, the utility of these asymptotic enumeration results is that allow us to easily `transfer' various quasirandomness properties, which depend only on a small number of rows of the adjacency matrix, from Erd\H{o}s-R\'enyi digraphs to uniform $d$-regular digraphs. 

\begin{definition}[Switching set]\label{def:switching-set}
For two vertices $i,j\in[n]$, we define their \emph{switching set} $S_{i,j}$ in digraph $A$ as the set of indices $k$ with $a_{ik}\neq a_{jk}$. Define the \emph{weight} of the switching set on a subset $S\subseteq[n]$ to be
\[\omega_{i,j}(S) = \sum_{k\in S}(a_{ik}-a_{jk}) = \sum_{k\in S\cap S_{i,j}}(a_{ik}-a_{jk}).\]
Note that $\omega_{i,j}([n]) = 0$ for a $d$-regular digraph $A$. 
\end{definition}

We now define a few events for a $d$-regular digraph $A$. 
\begin{definition}[Quasirandomness properties]
\label{def:Q-Q'-Q''}
For $A \in \mc{M}_{n,d}$, we define the following events.
\begin{enumerate}[(P1)]
\item Given $h \in \mb{N}$, let $\mc{Q}_h$ be the event that for any $2h$ distinct rows $A_{i_1},\ldots,A_{i_h}$ and $A_{j_1},\ldots,A_{j_h}$, we have
\[\bigg|\bigcap_{k=1}^hS_{i_k,j_k}^c\bigg|\le 2(\lambda^2+(1-\lambda)^2)^hn.\]
\item For $S\subseteq[n]$, let $\mc{Q}_S'$ be the event that for all sets of $4$ distinct rows $A_{i_1},A_{i_2},A_{j_1},A_{j_2}$, 
\[\min(|S_{i_1,j_1}\cap S_{i_2,j_2}\cap S|,|S_{i_1,j_1}\cap S_{i_2,j_2}\cap S^c|)\ge (2\lambda(1-\lambda))^2n/4.\]
\item For $S\subseteq [n]$, let $\mc{Q}_S''$ be the event that for every pair of distinct rows $A_i,A_j$, we have
\[|\omega_{i,j}(S)|\le\min\bigg(\frac{|S\cap S_{i,j}|}{6},\frac{|S^c\cap S_{i,j}|}{6}\bigg).\]
\item Finally, for a family $\mc{R}$ of subsets of $[n]$ and $h \in \mb{N}$, define
\[\mc{Q}_{h,\mc{R}} = \mc{Q}_h\cap\bigcap_{S\in\mc{R}}(\mc{Q}_S'\cap\mc{Q}_S'');\]
this final event encapsulates all the necessary quasirandomness conditions that we will need.
\end{enumerate}
\end{definition}
\begin{theorem}[Random regular digraphs are  quasirandom]\label{thm:quasirandom}
Let $h < n^{1/4}$ be a positive integer, and $\mc{R}\subseteq\binom{[n]}{n/2}$ be a family of sets. Let $A$ be chosen uniformly at random from $\mc{M}_{n,d}$. Then
\[\mb{P}[\mc{Q}_{h,\mc{R}}^c]\lesssim|\mc{R}|\exp(-\Omega_\lambda(n)).\]
\end{theorem}
\begin{remark}
In our application, $h$ will be a sufficiently large constant depending on various parameters (which in turn depend on $\lambda$); see \cref{eqn:eta-Q}. 
\end{remark}
\begin{proof}
A special case of \cite[Theorem~1]{CGM08} gives the following: let $N_{c}$ (respectively $N_{c'}$) denote the number of $(n-2h)\times n$ matrices with row sums $d = \lambda n$ and column sums $c_1,\ldots,c_n$ (respectively $c_1',\ldots, c_n'$) where $c_i, c_i'\in[d-2h,d]$. Then, $\max\{N_{c}/N_{c'}, N_{c'}/N_{c} \} \leq \exp(O_\lambda(h))$, for all $n$ sufficiently large (in terms of $\lambda$). 

In particular, the following is immediate: let $\mc{E}$ be an event for digraphs depending on at most $2h$ specified rows, let $p$ denote the probability of $\mc{E}$ for a uniformly chosen random $d$-regular digraph, and let $p'$ denote the probability of $\mc{E}$ for a uniformly chosen $\{0,1\}$-matrix subject to each row having sum $d$. Then, $p \leq p'\exp(O_\lambda(h))$ for all $n$ sufficiently large (in terms of $\lambda$). Moreover, letting $p''$ be the probability of $\mc{E}$ for the model where each entry of the $2h$ specified rows is an i.i.d. copy of $\on{Ber}(\lambda)$, and each of the remaining rows is chosen independently from the uniform distribution on vectors in $\{0,1\}^{n}$ summing to $d$, we see by a simple conditioning argument that $p' \leq O(n\lambda(1-\lambda))^{h}p''$.

Finally, the requisite probability bounds for the last model follow from a straightforward application of Hoeffding's inequality and the union bound, at which point we can conclude by the above comparison argument.     
\end{proof}

\subsection{Invertibility with respect to a single vector}
The goal of this subsection is to show that for any fixed vector $x \in \mb{S}_{0}^{n-1}$, $\snorm{Ax}_{2} \gtrsim_{\lambda} \sqrt{n}$, except with exponentially small probability.  
\begin{lemma}[Invertibility with respect to a fixed sum-zero vector]\label{lem:single-vector}
Let $d = \lambda n$. There is an absolute constant $c_{\lambda} > 0$ for which the following holds. Let $A$ be chosen uniformly at random from $\mc{M}_{n,d}$. Then, 
\[\sup_{x\in \mb{S}_{0}^{n-1}}\mb{P}[\snorm{Ax}_2\le c_\lambda\sqrt{n}]\le 2e^{-c_\lambda n}.\]
\end{lemma}
\begin{proof}
To start we note that $\snorm{D(x)}_2^2 = n$. We denote the rows of $A$ by $A_i$, and the columns of $A$ by $A^{(i)}$. For indices $i\neq j$, let $S_{i,j}$ denote the switching set of rows $A_i$ and $A_j$, and let $S^{(i,j)}$ denote the switching set of columns $A^{(i)}$ and $A^{(j)}$ (i.e., the set of $k$ with $a_{ki} \neq a_{kj}$). Let $m = \lfloor n/2 \rfloor$. For $(\sigma,A)$ distributed uniformly in $\mf{S}_{[n]}\times \mc{M}_{n,d}$, let $\mc{G}$ be the sigma-algebra generated by $\sigma$ and the random variables given by the row sums $A_\sigma(1)+A_\sigma(2),A_\sigma(3)+A_\sigma(4),\ldots,A_{\sigma(2m-1)}+A_{\sigma(2m)}$ (so that if $n$ is odd, then $A_{\sigma(n)}$ is measurable with respect to $\mc{G}$). Note that conditioned on $\mc{G}$, 
each of the vectors $A_{\sigma(2i-1)}-A_{\sigma(2i)}$ for $i\in [m]$ is distributed uniformly on the set of vectors supported on the switching set $S_{\sigma(2i-1),\sigma(2i)}$ that have $\pm 1$ entries within the support and sum to $0$.

Let $\mc{E}_1$ denote the event that $|S^{(i,j)}|\gtrsim_{\lambda}n$ for every pair of distinct $i,j\in[n]$. Then, from \cref{thm:quasirandom} (and row-column symmetry), we know that $\mb{P}[\mc{E}_1^{c}] \leq \exp(-\Omega_{\lambda}(n))$. 

Next, for every pair of distinct rows $i,j$, we define their weight (with respect to the vector $x$) to be
\[w_{i,j} = \sum_{k,\ell\in S_{i,j}}(x_k-x_\ell)^2.\]
Then, we see that
\[\sum_{i,j}w_{i,j} \ge \sum_{k,\ell}\frac{|S^{(k,\ell)}|^2}{2}(x_k-x_\ell)^2,\]
since every configuration with $a_{ik}=a_{j\ell}=1-a_{i\ell}=1-a_{jk}$ is counted on the left, and the right is a clear lower bound for this quantity. In particular, on the event $\mc{E}_1$,  
we have
\[\sum_{i,j}w_{i,j}\gtrsim_{\lambda}n^2\snorm{D(x)}_2^2.\]
Furthermore, since  $w_{i,j}\le\snorm{D(x)}_2^2$, we find that on the event $\mc{E}_1$, there are at least $\Omega_\lambda(n^2)$ pairs of distinct $i,j\in[n]$ with $w_{i,j}\gtrsim_\lambda\snorm{D(x)}_2^2 = n$.


Let $\mc{E}_{2}$ denote the event (measurable with respect to $\mc{G}$) that at least $\Omega_{\lambda}(n)$ `good' pairs $(\sigma(2i-1), \sigma(2i))$ satisfy $w_{\sigma(2i-1), \sigma(2i)} \gtrsim_{\lambda} n$. Then, the above discussion, along with a similar argument as in the proof of \cref{lem:robust-split-matching} shows that $\Pr[\mc{E}_2^{c}] \leq \exp(-\Omega_{\lambda}(n))$.  
On the event $\mc{E}_2$, define $\mc{P}$ to be the set of indices $i \in [m]$ such that $(\sigma(2i-1),\sigma(2i))$ is a good pair.  

We now demonstrate anticoncentration of $(A_{\sigma(2i-1)}-A_{\sigma(2i)})\cdot x$ for $i\in\mc{P}$. Let $y$ be the length $|S_{\sigma(2i-1),\sigma(2i)}|$ vector of $\pm 1$ values in $(A_{\sigma(2i-1)}-A_{\sigma(2i)})|_{S_{\sigma(2i-1),\sigma(2i)}}$ (noting that the rest of the vector is deterministically $0$). It is sum $0$ and uniform on this slice. Consider the linear function $f(y) = (A_{\sigma(2i-1)}-A_{\sigma(2i)})\cdot x$. Then, the hypercontractivity of linear functions on the central slice of the Boolean hypercube (cf. \cite[Lemma~5.2]{FM19}) shows that there exists some absolute constant $C \geq 1$ for which 
\[\mb{E}[|f(y)|^4]\le C^4\mb{E}[f(y)^2]^2.\]
Then, setting $\lambda^2 = \mb{E}[f(y)^2]/2$, the Paley-Zygmund inequality in \cite[Lemma~3.5]{LPRT05} gives 
\[\mb{P}[|f(y)| > \lambda]\ge\frac{\mb{E}[f(y)^2]^2}{4\mb{E}[f(y)^4]}\ge\frac{1}{4C^4}.\]
Noting that 
\[2\lambda^2 = \mb{E}[f(y)^2] = \frac{w_{2i-1,2i}}{|S_{2i-1,2i}|-1}\gtrsim_\lambda 1\]
for $i\in\mc{P}$, it follows that there exists some $c'_\lambda > 0$ such that for all $i \in \mc{P}$,
\[\mb{P}[|f(y)| > c'_{\lambda}]\ge\frac{1}{4C^4}.\]

Finally, since $A_{\sigma(2i-1)} - A_{\sigma(2i)}$ are conditionally independent given $\mc{G}$, and since $\snorm{Ax}_2 \geq \sum_{i \in \mc{P}} ((A_{\sigma(2i-1)} - A_{\sigma(2i)})\cdot x)^{2}$, it follows from tensorization (cf. \cite[Lemma~2.2(2)]{rudelson2008littlewood}) that there exists a constant $c_\lambda > 0$ such that for any $G \in \mc{E}_2$,
$$\mb{P}[\snorm{Ax}_2 < c_\lambda \sqrt{n} | \mc{G} = G] \leq \exp(-c_\lambda n).$$
The desired conclusion now follows using the law of total probability, after noting that $\mb{P}[\mc{E}_2^{c}] \leq \exp(-\Omega_{\lambda}(n))$ and after possibly decreasing $c_\lambda > 0$. 
\end{proof}

\section{Rerandomization, Switching, and Quantile Combinatorial LCD}\label{sec:rerandomization}
In this section, we introduce our main new ingredients -- refined switching operations, and the quantile Combinatorial LCD (QCLCD).
\subsection{Rerandomization and switching}\label{sub:switching}
Fix $(S,\sigma) \in \binom{[n]}{n/2} \times \mf{S}_{[n]}$. 
For $A \in \mc{M}_{n,d}$ with rows $A_i$, let $R_i = A_{\sigma(i)}$ and let $r_i(S)$ (respectively $r_i(S^c)$) denote the sum of $R_i|_S$ (respectively $R_i|_{S^c}$).
\begin{definition}[Revealed information]
\label{def:F}
For $A$ chosen uniformly from $\mc{M}_{n,d}$, let $\mc{F}_{S,\sigma}$ denote the sigma-algebra generated by the collection of random variables 
$$\{r_i(S), r_i(S^c)\}_{i \in [n]} \cup \{(R_{2i-1}+R_{2i})|_S, (R_{2i}+R_{2i+1})|_{S^c}\}_{i \in [\lfloor (n-1)/2 \rfloor]} \cup \{R_1|_{S^c}, R_n|_{P}\},$$
where $P = S$ if $n$ is odd and $P = S^{c}$ if $n$ is even. 
\end{definition}

The key point is that conditioned on $\mc{F}_{S,\sigma}$, there is additional randomness in the form of each $(R_{2i-1}-R_{2i})|_S$ and $(R_{2i}-R_{2i+1})|_{S^c}$. Note that each of these vectors has many fixed $0$s and some random $\pm 1$ signs (constrained to have a fixed sum), and moreover, that the random $\pm 1$ signs occur precisely where the two rows have a switching set (in the sense of \cref{def:switching-set}), which is measurable given $\mc{F}_{S,\sigma}$. This demonstrates the nomenclature: the sets $S$ allow one to, in the remaining randomness, `switch' between having $01$ in $R_i$ and $10$ in $R_{i+1}$ to $10$ and $01$, respectively. 

We will also make use of the following sets. 
\begin{definition}[Support of remaining randomness]
\label{def:T}
With notation as above, and for each $i \in [\lfloor (n-1)/2 \rfloor]$, let $T_{2i-1} = S\cap S_{\sigma(2i-1),\sigma(2i)}$ (i.e., it is the subset of $S$ such that the entry of $(R_{2i-1}+R_{2i})|_S$ is $1$), and similarly, let $T_{2i} = S^c\cap S_{\sigma(2i),\sigma(2i+1)}$. 
Note that these are measurable with respect to $\mc{F}_{S,\sigma}$.
\end{definition}

We note that in the study of the singularity and smallest singular value of random $d$-regular digraphs, the idea of `injecting randomness' using such switching operations goes back to the work of Cook \cite{Cook17b}. The main difference in our switching operation is the introduction of $(\sigma,S) \in \mf{S}_{[n]}\times \binom{[n]}{n/2}$, which will ultimately be chosen from a family $\mc{R}_{\delta, \rho}$ satisfying the conclusion of \cref{lem:robust-split-matching}. As we will see in \cref{eqn:anti-y}, the presence of the set $S$ will ensure that the event of a vector having small image is the tensorization of $n- O(1)$ independent random walks concentrating in a small interval; the crucial point here is that for proving the conjecture of Cook, $n - O(1)$ cannot be replaced by $n - \omega(1)$, whereas the switching construction in \cite{Cook17b} would naively only provide $n/2$ independent random walks. The permutation $\sigma$ dictates the order in which we reveal rows, and its properties will be crucially used in \cref{sub:large-QCLCD} (see the averaging step there), to ensure that the first term in \cref{thm:main} is $\kappa \sqrt{n}$ as opposed to $\kappa n^{1/2 + c}$ for some $c > 0$.    

\subsection{Quantile Combinatorial LCD}\label{sub:QCLCD}
We introduce a notion of arithmetic structure of vectors, which removes the `very worst' CLCDs of certain restrictions of the given vector. 
\begin{definition}[Quantile CLCD (QCLCD)]\label{def:QCLCD}
Let $v\in\mb{R}^n$ and $t \in \mb{N}$. Given $t$ sets (possibly repeated) of coordinates $\mc{T} = \{\{T_1,\ldots,T_t\}\}$ and $\ell\in[t]$, we define the quantile combinatorial LCD or $\on{QCLCD}_{\ell,\alpha,\gamma}^{\mc{T}}(v)$ to be the $\ell$th smallest value in the multiset
\[\{\{\on{CLCD}_{\alpha,\gamma}(v|_{T_i}): i\in[t]\}\}.\]
\end{definition}
\begin{remark}
Our notion of $\on{QCLCD}$ can be modified in the obvious way to yield a notion of $\on{QLCD}$ for the standard LCD, which can, for instance, be used to study the simpler model of random $d$-regular digraphs, each of whose non-zero entries is independently replaced by a Rademacher random variable. 
\end{remark}

In the rest of this subsection, we show that $\on{QCLCD}$ is not too small if the family of sets $\mc{T}$ is `well-spread' and the vector is not almost constant.
\begin{definition}[Well-spread family]\label{def:well-spread}
For $Q,t \in \mb{N}, \eta \in (0,1)$, and $U \subseteq [n]$, we say that a multifamily $\mc{U}$ of sets of coordinates $U_i\subseteq U$ for $i\in[t]$ is $(Q,\eta)$-\emph{well-spread with respect to} $U$ if:
\begin{enumerate}[(W1)]
\item (compare with (P1) in \cref{def:Q-Q'-Q''}) for every $Q$ distinct indices $i_1,\ldots,i_Q$, we have
\[\bigg|U\setminus\bigcup_{j=1}^QU_{i_j}\bigg|\le\eta|U|, \quad \text{and}\]
\item (compare with (P2) in \cref{def:Q-Q'-Q''}) for every pair $i,j \in [t]\times [t]$, we have $|U_i\cap U_j|\ge\eta|U|$.
\end{enumerate}
\end{definition}
\begin{lemma}[Bi-spread vectors have large QCLCD for well-spread families]\label{lem:lower-bound-D(x)}
Let $S \in \binom{[n]}{n/2}$, and suppose that $\mc{T}_1$ is $(Q,\eta)$-well-spread with respect to $S$ and $\mc{T}_2$ is $(Q,\eta)$-well-spread with respect to $S^c$. Let $x \in \mb{S}^{n-1}$, and suppose that $x$ satisfies $\mc{J}_\nu(x,S)$. Then, for $\mc{T} = \mc{T}_1\cup\mc{T}_2$ (as a multifamily), we have 
\begin{enumerate}[(C1)]
\item There are at most $2Q$ sets $T \in \mc{T}$ (with multiplicity) for which 
\begin{align}
\label{eqn:restriction-constant}
\snorm{D(x|_T)}_2\lesssim_{\nu,\eta,Q}\sqrt{n}\quad \on{or}\quad x|_T\in\on{Cons}_{\delta',\rho'};
\end{align}
\item $\on{QCLCD}_{2Q,\alpha,\gamma}^{\mc{T}}(x)\gtrsim_{\nu,\eta,Q}\sqrt{n}$,
\end{enumerate}
as long as $\eta\ll_\nu 1$, $\delta',\rho'\ll_{\nu,\eta,Q} 1$, $\gamma \ll _{\nu, \eta, Q} 1$.
\end{lemma}

\begin{proof}
Suppose for the sake of contradiction that (C1) is false.  By the pigeonhole principle, at least $Q$ of the sets (with multiplicity) satisfying \cref{eqn:restriction-constant} lie in $S$ or $S^c$; without loss of generality assume that $Q$ of these sets lie in $S$. Then, since $S$ has at least $\nu_{1}n$ indices $j$ for which $x_j \in [\nu_{2}/\sqrt{n}, \nu_{3}/\sqrt{n}]$, it follows from (W1) of \cref{def:well-spread} that for $\eta\le\nu_1/2$, at least one of the $Q$ sets satisfying \cref{eqn:restriction-constant} has $\nu_1n/(2Q)$ positive coordinates between $[\nu_2/\sqrt{n},\nu_3/\sqrt{n}]$. A similar argument shows that at least one of the $Q$ sets satisfying \cref{eqn:restriction-constant} has $\nu_1 n/(2Q)$ negative coordinates between $[-\nu_3/\sqrt{n},-\nu_2/\sqrt{n}]$. Consider the common intersection of these two sets, which by (W2) of \cref{def:well-spread} has size at least $\eta n$, and note that this intersection has either $\eta n/2$ nonnegative coordinates or $\eta n/2$ negative coordinates. Without loss of generality, suppose that there are at least $\eta n/2$ nonnegative coordinates. But then, for $T$ being the set with at least $\nu_{1} n/(2Q)$ coordinates between $[-\nu_3/\sqrt{n},-\nu_2/\sqrt{n}]$, we see that $\snorm{D(x|_{T})}_{2} \geq \sqrt{\nu_{1}\nu_{2}\eta/2Q}\cdot\sqrt{n}$ (and also, $x|_{T}$ is clearly not in $\on{Cons}_{\delta', \rho'}$)  which contradicts that $T$ satisfies \cref{eqn:restriction-constant}.

Finally, for (C2), note that for every set $T \in \mc{T}$ for which $x|_{T} \notin \on{Cons}_{\delta',\rho'}$, it follows from \cref{lem:lower-CLCD} and $\snorm{x|_{T}}_{2} \leq 1$ that $\on{CLCD}_{\alpha,\gamma}(x|_T)\gtrsim_{\delta',\rho'}\sqrt{N}$ as long as $\gamma \in (0,\delta'\rho'/12)$. Then, the conclusion follows immediately from (C1) and the definition of QCLCD.  
\end{proof}

\subsection{Nets for \texorpdfstring{$\on{QCLCD}$}{QCLCD}}\label{sub:nets-QCLCD}
In this subsection, we will construct sufficiently small nets for level sets of the QCLCD. 

\begin{definition}[Level sets of QCLCD]\label{def:level-set-QCLCD}
Fix a set system $\mc{T}$, an integer $Q \in \mb{N}$, $\nu=(\nu_1,\nu_2,\nu_3) \in \mb{R}^{3}$ with $\nu_i > 0$, $\mu \in (0,1)$, and $S\in\binom{[n]}{n/2}$. Suppose $H > 0$. We define
\[K_{\mc{T},H,\mu} = \{x\in\mb{S}^{n-1}: \mc{J}_\nu(x,S)\wedge H\le\on{QCLCD}_{2Q,\mu n,\gamma}^{\mc{T}}(x)\le 2H\}.\]
\end{definition}

Our goal is to show the following. 

\begin{lemma}[Nets for level sets of QCLCD]\label{lem:HS-net-QCLCD}
With notation as in \cref{def:level-set-QCLCD} and $\theta \in (0,1)$, suppose that $\mc{T}_1$ is $(Q,\eta)$-well-spread with respect to $S$ and $\mc{T}_2$ is $(Q,\eta)$-well-spread with respect to $S^c$, with each set in $\mc{T}=\mc{T}_1\cup\mc{T}_2$ of size at least $2\theta n$.
Assume that $0 < \delta,\rho\ll 1$, $0 < \mu\ll_{\delta,\rho}\gamma\ll_{\delta,\rho}1$, and $H\gtrsim_{\delta,\rho,\gamma}\sqrt{n}$. Then, there exists a collection $\mc{N}\subseteq K_{\mc{T},H,\mu}+(200\mu\sqrt{n}/H)B_2^n$ such that for every $x\in K_{\mc{T},H,\mu}$ and $m\times n$ matrix $B$, there is a point $y\in\mc{N}$ with
\[\snorm{B(x-y)}_2\le\frac{100\mu}{H}\snorm{B}_{\on{HS}},\]
and such that
\[|\mc{N}|\le H^3|\mc{T}|\bigg(\frac{C_{\delta,\rho,\gamma,\nu,\eta,Q, \theta}H\mu^{\theta-1}}{\sqrt{n}}\bigg)^n,\]
as long as $n$ is sufficiently large and $|\mc{T}| > 4Q$.
\end{lemma}
\begin{remark}
The critical point here is that -- compared to a $(200\mu \sqrt{n}/H)$-net obtained using the usual volumetric argument, which would have dependence $\mu^{n}$ in the size of the net -- the above net has the improved dependence $\mu^{(\theta - 1)n}$; this saving of $\mu^{\theta n}$ will be crucial for us. Another important point is the appearance of $\snorm{B}_{\on{HS}}$ (as opposed to the standard $\sqrt{n}\snorm{B}_{2}$), since the operator norm is `unusually large' compared to the Hilbert-Schmidt norm in our application (although this point can likely be bypassed, see the remark after \cref{thm:general-HS-net}).    
\end{remark}
The key ingredient 
in the proof of this lemma is the following randomized-rounding based net construction due to Livshyts \cite{Liv18}. 
\begin{theorem}[Specialization of {\cite[Theorem~4]{Liv18}}]\label{thm:general-HS-net}
There exists an absolute constant $C_{\ref{thm:general-HS-net}} > 0$ for which the following holds. Fix $\alpha\in(0,1/2)$ and $\beta\in(0,\alpha/10)$. Consider any $K\subseteq\mb{S}^{n-1}$ and $n\ge 1/\alpha^2$. Then, there exists a deterministic net $\mc{N}\subseteq K + (4\beta/\alpha)B_2^n$ such that for every $x\in K$ and $m\times n$ matrix $B$, there is a point $y \in \mc{N}$ with 
\[\snorm{B(x-y)}_2\le\frac{2\beta}{\alpha\sqrt{n}}\snorm{B}_{\on{HS}},\]
and such that 
\[|\mc{N}|\le N(K,\beta B_2^n)\exp(C\alpha^{0.08}\log(1/\alpha)n),\]
where $N(K,\beta B_2^n)$ is the covering number of the set $K$ with balls of radius $\beta$.
\end{theorem}
\begin{remark}
In \cite{Liv18}, the above statement is proved with $\snorm{B}_{\on{HS}}$ replaced by a certain regularized Hilbert-Schmidt norm (which is always at most the standard Hilbert-Schmidt norm), and in fact, a considerable amount of the effort in \cite{Liv18} is devoted to obtaining this more refined quantity on the right hand side. For our application, this is unnecessary since all matrices $B$ to which we will need to apply \cref{thm:general-HS-net,lem:HS-net-QCLCD} are $\{0,1\}$-valued, and hence, have $\snorm{B}_{\on{HS}} \leq \sqrt{mn}$ -- in particular, this permits a much more streamlined proof (using the techniques in \cite{Liv18}) of \cref{thm:general-HS-net} than the general \cite[Theorem~4]{Liv18}. 
We also note that one can replace the use of \cref{thm:general-HS-net} with a spectral gap estimate (as in \cite{jain2019approximate,Tra20}) for $d$-regular digraphs, which can likely be derived from more recent and refined asymptotic enumeration results due to Barvinok and Hartigan \cite{BH13}; this approach is substantially more technical and hence we have decided to use \cite{Liv18} instead.
\end{remark}


\begin{proof}[Proof of \cref{lem:HS-net-QCLCD}]
Let $\beta = (20\mu\sqrt{n}/H)$. We will bound $N(K_{\mc{T},H,\mu}, \beta B_{2}^{n})$, at which point the result will follow immediately from \cref{thm:general-HS-net} (with $\alpha = 1/3$). In order to do this, we will construct a $\beta$-net for $K_{\mc{T}, H, \mu}$ and bound its size. 

If $x\in K_{\mc{T},H,\mu}$, then by definition, at least $|\mc{T}|-(2Q-1)$ of the sets $T\in\mc{T}$ have
\[\on{CLCD}_{\mu n,\gamma}(x|_T)\in[H,2H].\]
Moreover, by \cref{lem:lower-bound-D(x)}, at least $|\mc{T}|-2(2Q-1)$ of these $|\mc{T}| - (2Q-1)$ sets $T$ additionally satisfy
\[\snorm{x}_2\sqrt{n}\ge\snorm{D(x|_T)}_2\gtrsim_{\nu,\eta,Q}\sqrt{n}.\]
Since $|\mc{T}| > 4Q$, we can choose $T\in\mc{T}$ satisfying both of the above equations. For the rest of the proof, fix such a set $T \in \mc{T}$; at the end, we will introduce an overall multiplicative factor of $|\mc{T}|$ in the size of the net to account for this choice. 

We note that by \cref{lem:metric-entropy-CLCD} applied with $\chi$ and $\zeta$ constants depending on $\nu,\eta,Q$, there is a $(9\mu\sqrt{|T|}/H)$-net for $x|_T$ of size at most
\[\mu^{-3}H^3\bigg(\frac{C_{\delta,\rho,\gamma,\nu,\eta,Q}H}{\sqrt{|T|}}\bigg)^{|T|}.\]
We take a $(9\mu\sqrt{n}/H)$-net of $B_2^{n-|T|}$ for $x|_{T^c}$ (with size bounded by the standard volumetric argument), and then take the product net, which has size at most
\[\mu^{-3}H^3\bigg(\frac{C_{\delta,\rho,\gamma,\nu,\eta,Q}H}{\sqrt{|T|}}\bigg)^{|T|}\times\bigg(\frac{4H}{9\mu\sqrt{n}}\bigg)^{n-|T|}\lesssim H^3\bigg(\frac{C_{\delta,\rho,\gamma,\eta,Q, \theta}H\mu^{\theta-1}}{\sqrt{n}}\bigg)^n.\]
In the last step, we used $n\ge |T| > 2\theta n$ and absorbed the $\mu^{-3}$ term into the exponential. The result now follows as indicated in the first paragraph of the proof. 
\end{proof}

\section{Singular value bound -- Proof of \cref{thm:main}}\label{sec:singular-value}
\subsection{Initial reduction}
Note that the vector $(1/\sqrt{n}, 1/\sqrt{n},\dots, 1/\sqrt{n})$ is deterministically a unit vector achieving the largest singular value; hence, the singular vector attaining the smallest singular must be orthogonal to it, so that we may restrict ourselves to $\mb{S}_0^{n-1}$ in the subsequent discussion. In particular, we fix maps $x \colon \mc{M}_{n,d} \to \mb{S}_{0}^{n-1}$ and $y: \mc{M}_{n,d} \to \mb{S}_{0}^{n-1}$ such that for $A \in \mc{M}_{n,d}$, $x(A)$ is a right least singular vector and $y(A)^{T}$ is a left least singular vector.

Throughout, we fix $\kappa$ as in \cref{thm:main}. Let $\mc{S}$ be the event that $\snorm{Ax(A)}_2\le\kappa$, which is the principal event we wish to study. Let $\chi > 0$ be a sufficiently small constant to be determined at the end of the analysis (this should not be confused with the abstract parameter $\chi$ appearing in \cref{def:level-set-CLCD}, \cref{lem:metric-entropy-CLCD}). We will assume that $\kappa\ge e^{-\chi n}$, since the statement of \cref{thm:main} for $\kappa < e^{-\chi n}$ follows from the statement for $\kappa = e^{-\chi n}$. 

Our proof will involve various parameters; the dependencies between them may be succinctly represented as follows: 
\begin{align}
\label{eqn:parameter-dependence}
(n^{-1}\alpha:=)\mu\ll\gamma \ll \eta,Q^{-1}\ll\nu_1,\nu_2,\nu_3\ll\delta,\rho\ll\lambda,
\end{align}
with $\mu$ chosen at the very end to enable various union bound arguments to go through with exponential room (technically, $\chi$ is chosen after $\mu$ but this is conceptually unimportant). More precisely, $\lambda$ is fixed in the statement of \cref{thm:main}. We choose  $\delta, \rho$ (depending only on $\lambda$) as in \cref{lem:compressible-singular-vector} below. Next, we choose $\nu = (\nu_1,\nu_2,\nu_3)$ as in \cref{lem:robust-split-matching}, based on $\delta,\rho$. This also gives us a family $\mc{R} = \mc{R}_{\delta,\rho}$ of pairs $(\sigma,S) \in \mf{S}_{[n]}\times \binom{[n]}{n/2}$ with certain properties that we will need. Note that $|\mc{R}| = O_{\delta, \rho}(1)$, and hence $O_\lambda(1)$ under the choices we have made.
Next, choose $Q$ and $\eta$ such that
\begin{align}
\label{eqn:eta-Q}
\eta < \lambda^{2}(1-\lambda)^{2}, \quad 2(\lambda^2+(1-\lambda)^2)^Q:= \eta \ll_\nu 1,
\end{align}
with the requisite smallness coming from \cref{lem:lower-bound-D(x)}. Having chosen $\eta, Q, \nu$, we choose $\gamma$ sufficiently small as per \cref{lem:lower-bound-D(x)}. Finally, we will work with the $\on{QCLCD}$ as in \cref{def:QCLCD} with parameter $\alpha = \mu n$, where $\mu$ will be taken to be a constant much smaller than all previously defined constants in accordance with \cref{eqn:small-LCD}. 

For the reader's convenience, we collect various events that will appear during the course of our proof. 
\begin{align*}
& \mc{S} = \{\snorm{Ax(A)}_{2} \leq \kappa\},\\
&\mc{C}_R = \{\exists x\in\on{Comp}_{\delta,\rho}^0: \snorm{Ax}_2 = \snorm{Ax(A)}_2\},\quad\mc{C}_L = \{\exists y\in\on{Comp}_{\delta,\rho}^0: \snorm{y^TA}_2 = \snorm{Ax(A)}_2\},\\
&\mc{C}=\mc{C}_L\cup\mc{C}_R,\\
&\mc{Q}_{Q,\mc{R}}\text{ as in \cref{sub:quasirandom}},\\
&\mc{I}_{\nu}(y(A), \sigma)\text{ as in \cref{def:I}},\\ &\mc{J}_{\nu}(x(A), S)\text{ as in \cref{def:J}},\\
&\mc{F}_{S, \sigma}\text{ is the sigma-algebra in \cref{def:F}}.
\end{align*}

We will also repeatedly abuse notation by stating expectation of events; events should be understood as the appropriate indicator. 

With these preliminaries, note that we have
\begin{align*}
\mb{P}[\mc{S}]&\le\mb{P}[\mc{C}\cap\mc{S}]+\mb{P}[\mc{Q}_{Q,\mc{R}}^c]+\mb{P}[\mc{C}^c\cap\mc{Q}_{Q,\mc{R}}\cap\mc{S}]\\
&\le O_\lambda(\exp(-cn)) + \mb{P}[\mc{Q}_{Q,\mc{R}}\cap\mc{C}^c\cap\mc{S}],
\end{align*}
where $c$ is the smaller of the two constants found in \cref{thm:quasirandom} and \cref{lem:compressible-singular-vector} below (note that this application of \cref{lem:compressible-singular-vector} requires $\kappa \lesssim _{\lambda} \sqrt{n}$, which we may assume without loss of generality, since \cref{thm:main} is trivially true outside this regime). 

Now
\[\mb{P}[\mc{Q}_{Q,\mc{R}}\cap\mc{C}^c\cap\mc{S}]\le\sum_{(\sigma,S)\in\mc{R}}\mb{P}[\mc{Q}_{Q,\mc{R}}\cap\mc{C}^c\cap\mc{S}\cap\mc{I}_\nu(y(A),\sigma)\cap\mc{J}_\nu(x(A),S)];\]
this holds since $\mc{C}^{c}$ guarantees that $x(A), y(A) \in \on{Incomp}_{\delta, \rho}^{0}$, so that by \cref{lem:robust-split-matching}, the events $\mc{I}_{\nu}(y(A), \sigma)$ and $\mc{J}_{\nu}(x(A), S)$ must hold for some choice of $(\sigma, S) \in \mc{R}$.
Since $|\mc{R}| = O_{\lambda}(1)$ by \cref{lem:robust-split-matching}, it follows that up to losing an overall multiplicative factor of $O_{\lambda}(1)$, we may (and will) restrict our attention to a fixed choice of $(\sigma, S) \in \mc{R}$ i.e., we will provide a uniform (in $(\sigma, S)$ upper bound on each summand on the right hand side of the above equation). 

Therefore, fix $(\sigma, S) \in \mc{R}$ and for $i \in [\lfloor (n-1)/2 \rfloor ]$, recall the definition of $T_{2i-1}, T_{2i}$ from \cref{def:T}. 
Let $\mc{T}_1$ be the multifamily of the odd-indexed sets $T_{2i-1}$ and $\mc{T}_2$ be the multifamily of the even-indexed sets $T_{2i}$. Then, by the law of total probability, we have
\begin{align*}
&\mb{P}[\mc{Q}_{Q,\mc{R}}\cap\mc{C}^c\cap\mc{S}\cap\mc{I}_\nu(y(A),\sigma)\cap\mc{J}_\nu(x(A),S)]\\
&= \mb{E}_{\mc{F}_{S,\sigma}}[\mb{P}[\mc{Q}_{Q,\mc{R}}\cap\mc{C}^c\cap\mc{S}\cap\mc{I}_\nu(y(A),\sigma)\cap\mc{J}_\nu(x(A),S)|\mc{F}_{S,\sigma}]].
\end{align*}
We will provide an upper bound on the inner probability which is uniform over the realisations of $\mc{F}_{S,\sigma}$.

Note that by the parameter choice in \cref{eqn:eta-Q}, it follows that on the event  $\mc{Q}_{Q,\mc{R}}$,  $\mc{T}_1$ is $(Q,\eta)$-well-spread with respect to $S$ (recall \cref{def:well-spread}) and $\mc{T}_2$ is $(Q,\eta)$-well-spread with respect to $S^c$. Thus, on the event $\mc{Q}_{Q,\mc{R}}\cap  \mc{J}_\nu(x(A),S)$, it follows from \cref{lem:lower-bound-D(x)} that
\[\snorm{D(x(A)|_T)}_2\gtrsim_{\nu,\eta,Q}\sqrt{n}\]
for all but less than $2Q$ sets $T\in\mc{T}$, and hence, from \cref{lem:lower-bound-D(x)} that
\[\on{QCLCD}_{2Q,\alpha,\gamma}^{\mc{T}}(x(A))\gtrsim_{\nu,\eta,Q}\sqrt{n}.\]
for all but less than $2Q$ sets $T \in \mc{T}$. Therefore, letting $D = 2^d$,
\begin{align}
\label{eqn:split-lcd}
&\mb{P}[\mc{Q}_{Q,\mc{R}}\cap\mc{C}^c\cap\mc{S}\cap\mc{I}_\nu(y(A),\sigma)\cap\mc{J}_\nu(x(A),S)|\mc{F}_{S,\sigma}] \nonumber \\
&\le\sum_{d=\log(c_{\nu, \eta, Q}\sqrt{n})}^{\log(\mu n/\kappa)}\mb{P}[\mc{Q}_{Q,\mc{R}}\cap\mc{C}^c\cap\mc{S}\cap\mc{I}_\nu(y(A),\sigma)\cap\mc{J}_\nu(x(A),S)\cap\on{QCLCD}_{2Q,\mu n,\gamma}^{\mc{T}}(x(A))\in[D,2D]|\mc{F}_{S,\sigma}] \nonumber\\
&\qquad+\mb{P}[\mc{Q}_{Q,\mc{R}}\cap\mc{C}^c\cap\mc{S}\cap\mc{I}_\nu(y(A),\sigma)\cap\mc{J}_\nu(x(A),S)\cap\on{QCLCD}_{2Q,\mu n,\gamma}^{\mc{T}}(x(A))\ge\mu n/\kappa|\mc{F}_{S,\sigma}],
\end{align}
where $c_{\nu, \eta, Q}$ is a constant depending on $\nu, \eta, Q$ coming from \cref{lem:lower-bound-D(x)}. We will deal with the first term in \cref{sub:small-QCLCD} and the second term in \cref{sub:large-QCLCD}. 

\subsection{Compressible vectors}
In this short subsection, we quickly show that $\mb{P}[\mc{C} \cap \mc{S}]$ is exponentially small. In fact, the following is a much stronger statement. 
\begin{lemma}\label{lem:compressible-singular-vector}
There exist $\delta,\rho,c\in(0,1)$ (depending only on $\lambda$) so that
\[\mb{P}\bigg[\inf_{x\in\on{Comp}_{\delta,\rho}^0}\snorm{Ax}_2 < c\sqrt{n}\bigg]\le 2\exp(-cn).\]
\end{lemma}
\begin{proof}
As is by now standard, the proof follows easily from combining the estimate for invertibility with respect to a single vector (\cref{lem:single-vector}) with an appropriate net argument (\cref{thm:general-HS-net}). We provide the details for completeness.

By \cref{lem:single-vector}, for any fixed $x\in\mb{S}_0^{n-1}$,
$\mb{P}[|Ax|\le c_{\lambda}\sqrt{n}]\le 2e^{-c_{\lambda}n}$.
Now we choose $\alpha_{\ref{thm:general-HS-net}}$ sufficiently small in terms of $c_\lambda$ so that
$-C_{\ref{thm:general-HS-net}}\alpha_{\ref{thm:general-HS-net}}^{0.08}\log(\alpha_{\ref{thm:general-HS-net}})\le c_{\lambda}/4$
and $\beta_{\ref{thm:general-HS-net}}$ sufficiently small in terms of $\alpha_{\ref{thm:general-HS-net}}$ so that
$2\beta_{\ref{thm:general-HS-net}} \le c_\lambda \alpha_{\ref{thm:general-HS-net}}/2$ (recall that $C_{\ref{thm:general-HS-net}} > 0$ is an absolute constant) Then, we choose $\delta,\rho$ sufficiently small so that $N(\on{Comp}_{\delta,\rho}^0,\beta_{\ref{thm:general-HS-net}} B_2^n)\le e^{c_{\lambda}n/4}$ (which is easily seen to be possible).

Applying \cref{thm:general-HS-net} to $S = \on{Comp}_{\delta,\rho}^0$ and $\alpha_{\ref{thm:general-HS-net}},\beta_{\ref{thm:general-HS-net}}$, there is a net $\mc{N}$ of size at most $e^{c_\lambda n/4}\cdot e^{c_\lambda n/4} = e^{c_\lambda n/2}$
such that for any $m\times n$ matrix $B$ with $\snorm{B}_{\on{HS}}\le n$ and $x\in\on{Comp}_{\delta,\rho}^0$, there is $y\in\mc{N}$ with
$\snorm{B(x-y)}_2\le\frac{c_{\lambda}\sqrt{n}}{2}.$

Since $\snorm{A}_{\on{Hs}} \leq n$ for all $A \in \mc{M}_{n,d}$, it therefore immediately follows that 
\begin{align*}
\mb{P}\bigg[\inf_{x\in\on{Comp}_{\delta,\rho}^0}\snorm{Ax}_2 < \frac{c_\lambda \sqrt{n}}{2}\bigg]&\le\mb{P}\bigg[\exists y\in\mc{N}: \snorm{Ay}_2 < c_\lambda \sqrt{n}\bigg]\le 2^{c_\lambda n/2}(2e^{-c_\lambda n}) = 2e^{-c_\lambda n/2}. \qedhere
\end{align*}
\end{proof}


\subsection{Small QCLCD}\label{sub:small-QCLCD}
In this subsection, we will bound the first term on the right hand side in \cref{eqn:split-lcd}, by showing that each summand is exponentially small. Thus, fix
$D\in[c_{\nu, \eta, Q}\sqrt{n},\mu n/\kappa]$. 
Then, recalling the definition of the level sets of the $\on{QCLCD}$, denoted by $K_{\mc{T}, D, \mu}$ (\cref{def:level-set-QCLCD}), we have
\begin{align*}
&\mb{P}[\mc{Q}_{Q,\mc{R}}\cap\mc{C}^c\cap\mc{S}\cap\mc{I}_\nu(y(A),\sigma)\cap\mc{J}_\nu(x(A),S)\cap\on{QCLCD}_{2Q,\mu n,\gamma}^{\mc{T}}(x(A))\in[D,2D]|\mc{F}_{S,\sigma}]\\
&\le\mb{P}[\mc{Q}_{Q,\mc{R}}\cap\snorm{Ax(A)}_2\le\kappa\cap x(A)\in K_{\mc{T},D,\mu}|\mc{F}_{S,\sigma}].
\end{align*}
Note that on the event $\mc{Q}_{Q,\mc{R}}$, we have in particular that $\mc{T}_1$ is $(Q,\eta)$-well-spread with respect to $S$ and $\mc{T}_2$ is $(Q,\eta)$-well-spread with respect to $S^c$, and also that each $T \in \mc{T}(:= \mc{T}_1 \cup \mc{T}_2)$ has size at least $2\theta n$, where $2\theta = \lambda^{2}(1-\lambda)^{2}$.\\  
\textbf{Approximation by a net: }Applying \cref{lem:HS-net-QCLCD}, we find that there is a net $\mc{N} \subseteq K_{\mc{T}, D, \mu} + (200\mu \sqrt{n}/D)B_2^{n}$ such that every $x\in K_{\mc{T},D,\mu}$ and every $m\times n$ matrix $A$ with $\snorm{A}_{\on{HS}} \leq n$, there exists a $y\in\mc{N}$ with
\[\snorm{A(x-y)}_2\le 100 \mu n /D.\]
Moreover, \cref{lem:HS-net-QCLCD} guarantees that 
\[|\mc{N}|\le D^3|\mc{T}|\bigg(\frac{CD\mu^{\theta-1}}{\sqrt{n}}\bigg)^n,\]
where $C$ depends only on $\delta,\rho,\gamma,\nu,\eta,Q, \theta$. 

\noindent \textbf{Anti-concentration of net points: }By definition of $\mc{N}$, for every $y\in\mc{N}$, we have $z\in K_{\mc{T},D,\mu}$ such that $\snorm{y-z}_2\le 200\mu\sqrt{n}/D$. Moreover, since $z \in K_{\mc{T},D,\mu}$, it follows from \cref{def:level-set-QCLCD} and \cref{lem:lower-bound-D(x)} (noting the well-spread properties of $\mc{T}_1, \mc{T}_2$ that are guaranteed on the event $\mc{Q}_{Q,\mc{R}}$) that $\snorm{D(z|_{T})}_{2} \gtrsim_{\nu, \eta, Q} \sqrt{n}$ for all but at least $|\mc{T}|-2Q$ sets $T \in \mc{T}$. But then, for all such choices of $T$, we have   
\begin{align}
\label{eqn:y-z}
\snorm{y|_T-z|_T}_2\le\snorm{y-z}_2\le\frac{200\mu\sqrt{n}}{D}\le\frac{\gamma\snorm{D(z|_T)}_2}{5\sqrt{n}}
\end{align}
as long as $\mu\ll_{\nu,\eta,Q}\gamma$ (since $D \geq c_{\nu, \eta, Q}\sqrt{n}$), which we will be able to ensure. Moreover, since $z\in K_{\mc{T}, D, \mu}$, for at least $|\mc{T}| - 4Q$ sets $T \in \mc{T}$, \cref{eqn:y-z} holds, and also $\on{CLCD}_{\mu n, \gamma}(z|_{T}) \geq D$. 
Therefore, by \cref{lem:stability-CLCD}, for at least $|\mc{T}| - 4Q$ sets $T \in \mc{T}$, we have 
\[\on{CLCD}_{\mu n/2,\gamma/2}(y|_T)\ge\min\bigg(D,\frac{\mu n}{4\sqrt{n}\snorm{y|_T-z|_T}_{2}}\bigg)\ge\frac{D}{800}.\]

Let the exceptional set of indices $i \in [|\mc{T}|]$ for which $T_i$ does not satisfy this property be $Y$, with $|Y|\le 4Q$. Then, for all $i\notin Y$ and $i \geq 2$, we have from \cref{lem:levy-concentration-CLCD} that
\begin{align}
\label{eqn:anti-y}
\mc{L}(A_{\sigma(i)}\cdot y|\mc{F}_{S,\sigma},A_{\sigma(1)},\ldots,A_{\sigma(i-1)};\epsilon)\lesssim_{\gamma,\nu,\eta,Q}\epsilon+\frac{1}{D}+e^{-4\mu^2N/9},
\end{align}
where $N = |T_{i-1}|\gtrsim_\lambda n$. Let us be more explicit about this deduction. First, note that the only randomness left in the row $A_{\sigma(i)}$ corresponds to the choices of $0$ and $1$ in $A_{\sigma(i)}|_{T_{i-1}}$ and furthermore, the fraction of zeros versus ones is constrained to be \[t := \frac{1}{2} + \frac{w_{\sigma(i-1),\sigma(i)}(S)}{N}\in [1/3,2/3].\] To see this, note that $w_{\sigma(i-1),\sigma(i)}(S)$ counts the difference in the number of forced ones in $T_{i-1}$ for $A_{\sigma(i-1)}$ and $A_{\sigma(i)}$, and the sum of the number of ones in the two rows, when restricted to $T_{i-1}$ is $|T_{i-1}|$ by definition. The inclusion of $t$ in the interval $[1/3, 2/3]$ holds because of the quasi-randomness condition $\mc{Q}_{Q,\mc{R}}$ (specifically, $\mc{Q}_S''$). Therefore the random variable $A_{\sigma(i)}\cdot y$, conditioned on the given information, is a shift of some variable $W_{tN,v}$ (\cref{def:W}) where $N = |T_{i-1}|\gtrsim_\lambda n$ and $v = y|_{T_{i-1}}$ (and the shift corresponds to the inner product of the remaining coordinates of $A_{\sigma(i-1)}$ and $y$, which is fixed). Given this, the anticoncentration claim for $A_{\sigma(i)}\cdot y$ follows as the various additional conditions for \cref{lem:levy-concentration-CLCD} follow from the fact that $i\notin Y$.

Finally, it follows easily from \cref{eqn:y-z} that for $\mu \ll_{\nu, \eta, Q} 1$, which we will be able to ensure, there is a lower bound on the length of $D(y|_{T_{i-1}})$, dependent only on $\nu,\eta,Q$, for all $i\notin Y$. 

\noindent \textbf{Tensorization and union bound: }
From \cref{eqn:anti-y} applied with $\epsilon\ge\epsilon_0 = 1/D$, and noting that $e^{-4\mu^2N/9}\le 1/D$ since $\kappa\ge e^{-\chi n}$ (and $\chi\ll\mu$), we have for all $i\notin Y$, $i\geq 2$ that
\[\mc{L}(A_{\sigma(i)}\cdot y|\mc{F}_{S,\sigma},A_{\sigma(1)},\ldots,A_{\sigma(i-1)};\epsilon)\lesssim K_{\gamma,\nu,\eta,Q}\epsilon.\]
Therefore, a straightforward conditional version of the tensorization inequality \cite[Lemma~2.2(1)]{rudelson2008littlewood} shows that for an absolute constant $C > 0$,
\[\mb{P}\bigg[\snorm{Ay}_2\le\kappa+\frac{100\mu n}{D}\bigg]\le(CK)^{n-1-4Q}\bigg(\frac{\kappa}{\sqrt{n}}+\frac{100\mu\sqrt{n}}{D}\bigg)^{n-1-4Q}\le\bigg(\frac{C\mu\sqrt{n}}{D}\bigg)^{n-4Q-1},\]
using $\kappa\le\mu n/D$ and changing $C$ between the second and third quantities. Here we implicitly used that $n\ge 8Q$.

Finally, putting everything together, we have 
\begin{align}
\label{eqn:small-LCD}
&\mb{P}[\mc{Q}_{Q,\mc{R}}\cap\snorm{Ax(A)}_2\le\kappa\cap x(A)\in K_{\mc{T},D,\mu}|\mc{F}_{S,\sigma}] \nonumber \\
&\le\sum_{y\in\mc{N}}\mb{P}\bigg[\snorm{Ay}_2\le\kappa+\frac{100\mu n}{D}\bigg|\mc{F}_{S,\sigma}\bigg] \nonumber\\
&\le D^3|\mc{T}|\bigg(\frac{CD\mu^{\theta-1}}{\sqrt{n}}\bigg)^n\cdot\bigg(\frac{C\mu\sqrt{n}}{D}\bigg)^{n-4Q-1}\le D^{4Q+4}\mu^{-4Q-1}(C\mu^\theta)^n,
\end{align}
changing $C$ between the final two quantities. Note here that $C$ does not depend on $\mu$ (or $\chi$), and that $\theta > 0$ is fixed by the value of $\lambda$.  
Therefore, taking $\mu$ sufficiently small yields the desired result in this case, noting that we have an upper bound  $D\le\mu n/\kappa \leq \mu n e^{\chi n}$ and can choose $\chi$ sufficiently small depending on $\mu$.  

\subsection{Large QCLCD}\label{sub:large-QCLCD}
In this subsection, we will bound the second term on the right hand side of \cref{eqn:split-lcd}, i.e.,
\[\mb{E}_{\mc{F}_{S,\sigma}}[\mb{P}[\mc{Q}_{Q,\mc{R}}\cap\mc{C}^c\cap\mc{S}\cap\mc{I}_\nu(y(A),\sigma)\cap\mc{J}_\nu(x(A),S)\cap\on{QCLCD}_{2Q,\mu n,\gamma}^{\mc{T}}(x(A))\ge\mu n/\kappa|\mc{F}_{S,\sigma}]].\]
In this case, we will not be able to use a direct union bound argument as in the previous subsection, and will instead use a variant of an argument given in \cite{LLTTY19}, with the crucial addition of consideration of arithmetic structure. As before, we will provide an upper bound on the inner probability which is uniform over the realisations of $\mc{F}_{S,\sigma}$. 

\noindent \textbf{Averaging: }
For any $x\in\mb{S}_0^{n-1}$, we define on the event $\mc{J}_{\nu}(x,S) \cap \mc{Q}_{Q, \mc{R}}$ the set  $\on{Sm}^{\mc{T}}(x)$ to contain the at most $2Q$ (by \cref{lem:lower-bound-D(x)}) indices $i$ which satisfy $\snorm{D(x|_{T_i})}_2\lesssim_{\nu,\eta,Q}\sqrt{n}$ and the $2Q$ indices corresponding to the $2Q$ lowest values of $\on{CLCD}_{2Q,\mu n,\gamma}(x|_{T_i})$. In particular, $|\on{Sm}^{\mc{T}}(x)| \leq 4Q$. 

Thus, on the event $\mc{J}_{\nu}(x(A),S)\cap \mc{Q}_{Q,\mc{R}}$, we have by definition that 
\[\mbm{1}[\mc{I}_\nu(y(A),\sigma)]\le\frac{1}{\nu_1n-4Q}\sum_{i=1}^n\mbm{1}[|y(A)_{\sigma(i)}-y(A)_{\sigma(i+1)}|\sqrt{n}\ge\nu_2\cap i\notin\on{Sm}^{\mc{T}}(x(A))],\]
so that on the event $\mc{Q}_{Q, \mc{R}}$,
\begin{align*}
&\mbm{1}[\mc{I}_\nu(y(A),\sigma)\cap\mc{J}_\nu(x(A),S)\cap\on{QCLCD}_{2Q,\mu n,\gamma}^{\mc{T}}(x(A))\ge\mu n/\kappa]\\
&\le\frac{1}{\nu_1n-4Q}\sum_{i=1}^{n-1}\mbm{1}[|y(A)_{\sigma(i)}-y(A)_{\sigma(i+1)}|\sqrt{n}\ge\nu_2\cap i\notin\on{Sm}^{\mc{T}}(x)\cap\on{QCLCD}_{2Q,\mu n,\gamma}^{\mc{T}}(x(A))\ge\mu n/\kappa]\\
&\le\frac{1}{\nu_1n-4Q}\sum_{i=1}^{n-1}\mbm{1}[|y(A)_{\sigma(i)}-y(A)_{\sigma(i+1)}|\sqrt{n}\ge\nu_2\cap\on{CLCD}_{\mu n,\gamma}^{\mc{T}}(x(A)|_{T_i})\ge\mu n/\kappa\\
&\qquad\qquad\qquad\qquad\qquad\cap\snorm{D(x(A)|_{T_i})}_2\gtrsim_{\nu,\eta,Q}\sqrt{n}].
\end{align*}

Using this, and taking probabilities gives 
\begin{align}
\label{eq:averaged-event}
&\mb{E}_{\mc{F}_{S,\sigma}}[\mb{P}[\mc{Q}_{Q,\mc{R}}\cap\mc{C}^c\cap\mc{S}\cap\mc{I}_\nu(y(A),\sigma)\cap\mc{J}_\nu(x(A),S)\cap\on{QCLCD}_{2Q,\mu n,\gamma}^{\mc{T}}(x(A))\ge\mu n/\kappa|\mc{F}_{S,\sigma}]] \nonumber \\
&\le\frac{1}{\nu_1n-4Q}\sum_{i=1}^{n-1}\mb{P}[\mc{Q}_{Q,\mc{R}}\cap\mc{C}^c\cap\mc{S}\cap\mc{I}_\nu(y(A),\sigma)\cap\mc{J}_\nu(x(A),S)\cap\on{QCLCD}_{2Q,\mu n,\gamma}^{\mc{T}}(x(A))\ge\mu n/\kappa \nonumber \\
&\qquad\cap|y(A)_{\sigma(i)}-y(A)_{\sigma(i+1)}|\sqrt{n}\ge\nu_2\cap\on{CLCD}_{\mu n,\gamma}(x(A)|_{T_i})\ge\mu n/\kappa\cap\snorm{D(x(A)|_{T_i})}_2\gtrsim_{\nu,\eta,Q}\sqrt{n}].
\end{align}
In the remainder of the proof, we will bound each of the $n-1$ probabilities in the above equation by a constant (depending on $\gamma, \eta, \nu, Q$) times $\kappa\sqrt{n}+\exp(-\Omega_\lambda(n))$, which will suffice. Without loss of generality, we will do this for the index $i=1$ (the argument for other indices follows by purely notational changes).

\noindent \textbf{Partitioning $\mc{M}_{n,d}$: }
We follow a similar idea as in \cite{LLTTY19} of partitioning our event space based on realisations of rows other than $A_{\sigma(1)}, A_{\sigma(2)}$. 

More precisely, let $\mc{H}$ be the set of all possible realizations of  $\mc{F}_{S,\sigma}$ as well as all elements other than $A_{\sigma(1)}|_{T_1},A_{\sigma(2)}|_{T_1}$. In particular, given an element in $H \in \mc{H}$, extending it to an element of $\mc{M}_{n,d}$ amounts to choosing the vector $(A_{\sigma(1)}-A_{\sigma(2)})|_{T_1}$, which is a $\pm 1$-valued vector with a fixed sum $w_{\sigma(1),\sigma(2)}(T_1)$ (note that the sum is fixed by $H$). For $H \in \mc{H}$, let $C_H$ be the subset of $\mc{M}_{n,d}$ extending $H$ in this manner. 
Let $G_H$ be the subset of $C_H$ satisfying
\[\on{CLCD}_{\mu n,\gamma}(x(M)|_{T_1})\ge\mu n/\kappa\cap\snorm{D(x(M)|_{T_1})}_2\gtrsim_{\nu,\eta,Q}\sqrt{n}.\]
Let $\mc{H}_0$ be the set of $H\in\mc{H}$ such that either $G_H=\emptyset$ or such that the realisation of $\mc{F}_{S,\sigma}$ determined by $H$ does not satisfy $\mc{Q}_{Q,\mc{R}}$. 
In particular, for each $H\in\mc{H}\setminus\mc{H}_0$, we have $|G_H|\ge 1$. Finally, for each $H \in \mc{H}\setminus \mc{H}_0$, let $\wt{M}_H$ be a fixed (but otherwise arbitrarily chosen) matrix in $G_H$ with smallest least singular value among all matrices in $G_H$. Then, by the definition of $G_H$, we have
\[\on{CLCD}_{\mu n,\gamma}(x(\wt{M}_H)|_{T_1})\ge\mu n/\kappa\cap\snorm{D(x(\wt{M}_H)|_{T_1})}_2\gtrsim_{\nu,\eta,Q}\sqrt{n}.\]

\noindent \textbf{Reduction to distance to subspace: }Noting that for $H \in \mc{H}_0$, no $M \in C_H$ can simultaneously satisfy all three events $\mc{Q}_{Q,\mc{R}}$ and $\on{CLCD}_{\mu n,\gamma}(x(M)|_{T_1})\ge\mu n/\kappa$ and $\snorm{D(x(M)|_{T_1})}_2\gtrsim_{\nu,\eta,Q}\sqrt{n}$ appearing in the probability on the right hand side of \cref{eq:averaged-event}, it suffices to bound
\begin{align}
\label{eqn:event-not-H_0}
&\mb{P}[H\notin\mc{H}_0\cap\mc{Q}_{Q,\mc{R}}\cap\mc{S}\cap|y(A)_{\sigma(1)}-y(A)_{\sigma(2)}|\sqrt{n}\ge\nu_2\cap\on{CLCD}_{\mu n,\gamma}(x(A)|_{T_1})\ge\mu n/\kappa \nonumber \\
&\qquad\qquad\qquad\cap\snorm{D(x(A)|_{T_1})}_2\gtrsim_{\nu,\eta,Q}\sqrt{n}].
\end{align}
Moreover, on the event $\mc{S}$ and $|y(A)_{\sigma(1)}-y(A)_{\sigma(2)}|\sqrt{n}\ge\nu_2$, we have
\[\kappa\ge \snorm{y(A)^TA}_2 = \bigg\|\sum_{i=1}^ny(A)_iA_i\bigg\|_{2}\ge\frac{|y(A)_{\sigma(1)}-y(A)_{\sigma(2)}|}{2}\on{dist}(A_{\sigma(1)}-A_{\sigma(2)},V),\]
where $V = \on{span}\{A_{\sigma(1)}+A_{\sigma(2)},A_k: k\notin\{\sigma(1),\sigma(2)\}\}$. Thus, on $S \cap |y(A)_{\sigma(1)}-y(A)_{\sigma(2)}|\sqrt{n}\ge\nu_2|$, we must have, 
\[\on{dist}(A_{\sigma(1)}-A_{\sigma(2)},V)\le\frac{2\kappa\sqrt{n}}{\nu_2},\]
so that the probability in \cref{eqn:event-not-H_0} is bounded above by
\begin{align}
\label{eqn:event-distance}
&\mb{P}[H\notin\mc{H}_0\cap\mc{Q}_{Q,\mc{R}}\cap\on{CLCD}_{\mu n,\gamma}(x(A)|_{T_1})\ge\mu n/\kappa\cap\snorm{D(x(A)|_{T_1})}_2\gtrsim_{\nu,\eta,Q}\sqrt{n} \nonumber \\
&\cap\on{dist}(A_{\sigma(1)}-A_{\sigma(2)},V)\le2\kappa\sqrt{n}/\nu_2].
\end{align}

\noindent \textbf{Anti-concentration: }At this point, note that if $x(A)$ were independent of $(A_{\sigma(1)} - A_{\sigma(2)})|_{T_1}$, and further, if it were the normal to $V$, then we would be able to use small-ball concentration of this relatively unstructured vector (\cref{lem:levy-concentration-CLCD}) to complete the proof. 
Unfortunately, we do not have these two properties. To overcome this problem, we use \cite[Lemma~4.3]{LLTTY19}, which allows us to use a vector `approximately normal' to the subspace $V$ in order to lower bound $\on{dist}(A_{\sigma(1)}-A_{\sigma(2)},V)$. 
\begin{lemma}[{\cite[Lemma~4.3]{LLTTY19}}]
\label{lem:dist}
With notation as above, letting $N$ denote the $(n-2)\times n$ matrix obtained by removing rows $\sigma(1),\sigma(2)$ from $A$, and for every $w \in \mb{S}^{n-1}$, we have
\begin{align*}
    \on{dist}(A_{\sigma(1)}, V) \geq \frac{s_n(A)|\sang{A_{\sigma(1)},w}|}{s_n(A) + \snorm{N w}_2 + |\sang{A_{\sigma(1)}+A_{\sigma(2)},w}|}.
\end{align*}
Hence, if $\snorm{Nw}_2 \leq s_n(A)$ and $|\sang{A_{\sigma(1)}+A_{\sigma(2)},w}| \leq 2s_n(A)$, then $\on{dist}(R_{\sigma(1)},V) \geq |\sang{A_{\sigma(1)},w}|/4$.
\end{lemma}
As we will see, one can take the vector $w$ in the above lemma to be $x(\wt{M}_H)$, which only depends on $H \in \mc{H}$, and hence, is independent of $(A_{\sigma(1)}-A_{\sigma(2)})|_{T_1}$.  
Indeed, note that any $A$ satisfying the event in \cref{eqn:event-distance} is in $G_H$ for some $H \in \mc{H}$, and that by definition, $s_n(\wt{M}_H)\le s_n(A)$. 
Let $N_H$ be the $(n-2)\times n$ matrix obtained by removing the rows $\sigma(1),\sigma(2)$ from $A$, and let $s_H$ be the vector $A_{\sigma(1)}+A_{\sigma(2)}(= (\wt{M}_H)_{\sigma(1)} + (\wt{M}_H)_{\sigma(2)})$; as the notation suggests, both $N_H$ and $s_H$ depend only on $H$. 
We have
\begin{align*}
&\snorm{N_Hx(\wt{M}_H)}_2\le\snorm{\wt{M}_Hx(\wt{M}_H)} = s_n(\wt{M}_H)\le s_n(A),\\
&|\sang{s_H,x(\wt{M}_H)}|\le|\sang{(\wt{M}_H)_{\sigma(1)}+(\wt{M}_H)_{\sigma(2)},x(\wt{M}_H)}|\le 2s_n(\wt{M}_H)\le 2s_n(A).
\end{align*}
Therefore, \cref{lem:dist} shows that 
$$\on{dist}(A_{\sigma(1)}, V) \geq |\sang{A_{\sigma(1)}, x(\wt{M}_H)}|/4.$$
Noting that $\on{dist}(A_{\sigma(1)}-A_{\sigma(2)}, V) = 2\on{dist}(A_{\sigma(1)},V)$, which is readily seen using $A_{\sigma(1)}+A_{\sigma(2)}\in V$, it follows from the above equation that the probability in \cref{eqn:event-distance} is bounded above by
\begin{align}
\label{eqn:event-final}
    \mb{P}[H\notin\mc{H}_0\cap\mc{Q}_{Q,\mc{R}}\cap|\sang{x(\wt{M}_H),A_{\sigma(1)}}|\le 4\kappa\sqrt{n}/\nu_2]
\end{align}
Finally, since by the definition of $G_H$, we have
\[\on{CLCD}_{\mu n,\gamma}(x(\wt{M}_H)|_{T_1})\ge\mu n/\kappa\cap\snorm{D(x(\wt{M}_H)|_{T_1})}_2\gtrsim_{\nu,\eta,Q}\sqrt{n},\]
it follows from \cref{lem:levy-concentration-CLCD} (in the same manner as \cref{eqn:anti-y}, provided that we first reveal $H$ and then look at the remaining randomness in $A_{\sigma(1)}$) that the probability in \cref{eqn:event-final} is bounded above by
\[\frac{4C\kappa\sqrt{n}}{\nu_2} + \frac{C\kappa}{\mu n} + Ce^{-\Omega_\lambda(n)},\]
where $C$ depends only on $\gamma, \eta, \nu, Q$. This completes the proof. 

\bibliographystyle{amsplain0.bst}
\bibliography{main.bib}

\providecommand{\bysame}{\leavevmode\hbox to3em{\hrulefill}\thinspace}
\providecommand{\MR}{\relax\ifhmode\unskip\space\fi MR }
\providecommand{\MRhref}[2]{%
  \href{http://www.ams.org/mathscinet-getitem?mr=#1}{#2}
}
\providecommand{\href}[2]{#2}
\begin{thebibliography}{10}

\bibitem{BH13}
Alexander Barvinok and J.~A. Hartigan, \emph{The number of graphs and a random
  graph with a given degree sequence}, Random Structures Algorithms \textbf{42}
  (2013), 301--348.

\bibitem{campos2019singularity}
Marcelo Campos, Let{\'\i}cia Mattos, Robert Morris, and Natasha Morrison,
  \emph{On the singularity of random symmetric matrices}, arXiv preprint
  arXiv:1904.11478.

\bibitem{CGM08}
E.~Rodney Canfield, Catherine Greenhill, and Brendan~D. McKay, \emph{Asymptotic
  enumeration of dense 0-1 matrices with specified line sums}, J. Combin.
  Theory Ser. A \textbf{115} (2008), 32--66.

\bibitem{cook2019circular}
Nicholas Cook, \emph{The circular law for random regular digraphs}, Ann. Inst.
  Henri Poincar\'{e} Probab. Stat. \textbf{55} (2019), 2111--2167.

\bibitem{Cook17b}
Nicholas~A. Cook, \emph{On the singularity of adjacency matrices for random
  regular digraphs}, Probab. Theory Related Fields \textbf{167} (2017),
  143--200.

\bibitem{costello2006random}
Kevin~P. Costello, Terence Tao, and Van Vu, \emph{Random symmetric matrices are
  almost surely nonsingular}, Duke Math. J. \textbf{135} (2006), 395--413.

\bibitem{edelman1988eigenvalues}
Alan Edelman, \emph{Eigenvalues and condition numbers of random matrices}, SIAM
  J. Matrix Anal. Appl. \textbf{9} (1988), 543--560.

\bibitem{ferber2019singularity}
Asaf Ferber and Vishesh Jain, \emph{Singularity of random symmetric
  matrices---a combinatorial approach to improved bounds}, Forum Math. Sigma
  \textbf{7} (2019), e22, 29.

\bibitem{ferber2019counting}
Asaf Ferber, Vishesh Jain, Kyle Luh, and Wojciech Samotij, \emph{On the
  counting problem in inverse {L}ittlewood--{O}fford theory}, arXiv:1904.10425.

\bibitem{FM19}
Yuval Filmus and Elchanan Mossel, \emph{Harmonicity and invariance on slices of
  the {B}oolean cube}, Probab. Theory Related Fields \textbf{175} (2019),
  721--782.

\bibitem{ge2017eigenvalue}
Stephen Ge, \emph{The eigenvalue spacing of iid random matrices and related
  least singular value results}, Ph.D. thesis, UCLA, 2017.

\bibitem{huang2018invertibility}
Jiaoyang Huang, \emph{Invertibility of adjacency matrices for random d-regular
  directed graphs}, arXiv:1806.01382.

\bibitem{jain2019approximate}
Vishesh Jain, \emph{Approximate {S}pielman-{T}eng theorems for the least
  singular value of random combinatorial matrices}, arXiv:1904.10592.

\bibitem{kahn1995probability}
Jeff Kahn, J\'{a}nos Koml\'{o}s, and Endre Szemer\'{e}di, \emph{On the
  probability that a random {$\pm 1$}-matrix is singular}, J. Amer. Math. Soc.
  \textbf{8} (1995), 223--240.

\bibitem{komlos1967determinant}
J.~Koml\'{o}s, \emph{On the determinant of {$(0,\,1)$} matrices}, Studia Sci.
  Math. Hungar. \textbf{2} (1967), 7--21.

\bibitem{LPRT05}
A.~E. Litvak, A.~Pajor, M.~Rudelson, and N.~Tomczak-Jaegermann, \emph{Smallest
  singular value of random matrices and geometry of random polytopes}, Adv.
  Math. \textbf{195} (2005), 491--523.

\bibitem{litvak2017adjacency}
Alexander~E. Litvak, Anna Lytova, Konstantin Tikhomirov, Nicole
  Tomczak-Jaegermann, and Pierre Youssef, \emph{Adjacency matrices of random
  digraphs: singularity and anti-concentration}, J. Math. Anal. Appl.
  \textbf{445} (2017), 1447--1491.

\bibitem{LLTTY17}
Alexander~E. Litvak, Anna Lytova, Konstantin Tikhomirov, Nicole
  Tomczak-Jaegermann, and Pierre Youssef, \emph{Adjacency matrices of random
  digraphs: singularity and anti-concentration}, J. Math. Anal. Appl.
  \textbf{445} (2017), 1447--1491.

\bibitem{LLTTY19}
Alexander~E. Litvak, Anna Lytova, Konstantin Tikhomirov, Nicole
  Tomczak-Jaegermann, and Pierre Youssef, \emph{The smallest singular value of
  a shifted {$d$}-regular random square matrix}, Probab. Theory Related Fields
  \textbf{173} (2019), 1301--1347.

\bibitem{Liv18}
Galyna~V Livshyts, \emph{The smallest singular value of heavy-tailed not
  necessarily iid random matrices via random rounding}, arXiv:1811.07038.

\bibitem{livshyts2019smallest}
Galyna~V Livshyts, Konstantin Tikhomirov, and Roman Vershynin, \emph{The
  smallest singular value of inhomogeneous square random matrices},
  arXiv:1909.04219.

\bibitem{MW90}
Brendan~D. McKay and Nicholas~C. Wormald, \emph{Asymptotic enumeration by
  degree sequence of graphs of high degree}, European J. Combin. \textbf{11}
  (1990), 565--580.

\bibitem{nguyen2013singularity}
Hoi~H. Nguyen, \emph{On the singularity of random combinatorial matrices}, SIAM
  J. Discrete Math. \textbf{27} (2013), 447--458.

\bibitem{nguyen2012circular}
Hoi~H. Nguyen and Van~H. Vu, \emph{Circular law for random discrete matrices of
  given row sum}, J. Comb. \textbf{4} (2013), 1--30.

\bibitem{rudelson2008littlewood}
Mark Rudelson and Roman Vershynin, \emph{The {L}ittlewood-{O}fford problem and
  invertibility of random matrices}, Adv. Math. \textbf{218} (2008), 600--633.

\bibitem{rudelson2016no}
Mark Rudelson and Roman Vershynin, \emph{No-gaps delocalization for general
  random matrices}, Geom. Funct. Anal. \textbf{26} (2016), 1716--1776.

\bibitem{szarek1991condition}
Stanis{\l}aw~J. Szarek, \emph{Condition numbers of random matrices}, J.
  Complexity \textbf{7} (1991), 131--149.

\bibitem{tao2010random}
Terence Tao and Van Vu, \emph{Random matrices: universality of {ESD}s and the
  circular law}, Ann. Probab. \textbf{38} (2010), 2023--2065, With an appendix
  by Manjunath Krishnapur.

\bibitem{tao2009inverse}
Terence Tao and Van~H. Vu, \emph{Inverse {L}ittlewood-{O}fford theorems and the
  condition number of random discrete matrices}, Ann. of Math. (2) \textbf{169}
  (2009), 595--632.

\bibitem{tikhomirov2020singularity}
Konstantin Tikhomirov, \emph{Singularity of random {B}ernoulli matrices}, Ann.
  of Math. (2) \textbf{191} (2020), 593--634.

\bibitem{Tra20}
Tuan Tran, \emph{The smallest singular value of random combinatorial matrices},
  arXiv:2007.06318.

\bibitem{vershynin2014invertibility}
Roman Vershynin, \emph{Invertibility of symmetric random matrices}, Random
  Structures Algorithms \textbf{44} (2014), 135--182.

\bibitem{vershynin2018high}
Roman Vershynin, \emph{High-dimensional probability}, Cambridge Series in
  Statistical and Probabilistic Mathematics, vol.~47, Cambridge University
  Press, Cambridge, 2018, An introduction with applications in data science,
  With a foreword by Sara van de Geer.

\end{thebibliography}

\end{document}